\documentclass{amsart}
\usepackage{amssymb, amsmath, amsfonts, amsthm, esint, mathtools}
\usepackage[dvipsnames,svgnames,table]{xcolor}
\usepackage[margin=1in]{geometry}
\usepackage[utf8]{inputenc}
\usepackage{cite}
\usepackage{comment}

\usepackage{tikz}
    \usetikzlibrary{calc, angles, patterns, quotes, arrows.meta}
    \usepackage{pgfplots}
    \pgfplotsset{compat=1.14}

\mathtoolsset{showonlyrefs}

\usepackage[colorlinks,linkcolor = blue,citecolor = ForestGreen]{hyperref}

\newcommand{\bb}{\mathfrak{b}}

\newcommand{\abs}[1]{\left\vert#1\right\vert}
\newcommand{\norm}[1]{\left\|#1\right\|}
\newcommand{\set}[1]{\left\{ #1 \right\}}
\newcommand{\brak}[1]{\left\langle #1 \right\rangle}
\newcommand{\vv}{\brak{v}}
\newcommand{\vvo}{\brak{v_0}}

\newcommand{\be}{\begin{equation}}
\newcommand{\ee}{\end{equation}}
\newcommand{\dd}{\, {\rm d}}
\newcommand{\eps}{\varepsilon}

\newcommand{\R}{\ensuremath{{\mathbb R}}}
\renewcommand{\S}{\mathbb S}
\newcommand{\T}{\ensuremath{{\mathbb T}}}
\DeclareMathOperator{\Id}{Id}
\DeclareMathOperator{\QL}{Q_{\rm L}}
\DeclareMathOperator{\QB}{Q_{\rm B}}
\DeclareMathOperator{\Q}{Q}
\DeclareMathOperator{\Qs}{Q_{s}}
\DeclareMathOperator{\Qns}{Q_{ns}}

\newcommand{\uf}{{\underline f}}

\newcounter{num} \numberwithin{num}{section}

\newtheorem{theorem}[num]{Theorem}

\newtheorem{lemma}[num]{Lemma}
\newtheorem{corollary}[num]{Corollary}
\theoremstyle{definition}
\theoremstyle{remark}
\newtheorem{remark}[num]{Remark}
\numberwithin{equation}{section}

\author{William Golding, Christopher Henderson, and Luis Silvestre}
\title[Pointwise bounds]{Pointwise bounds and obstructions to blowup for the Landau and Boltzmann equations}
\address[William Golding]
    {Department of Mathematics, The University of Chicago,
    Chicago, IL 60615, USA}
    \email{wgolding@uchicago.edu}
\address[Christopher Henderson]
    {Department of Mathematics, The University of Maryland,
    College Park, MD 20742, USA}
    \email{ckhend@umd.edu}
\address[Luis Silvestre]
    {Department of Mathematics, The University of Chicago,
    Chicago, IL 60615, USA}
    \email{luis@math.uchicago.edu}
\thanks{\textbf{Acknowledgments:} The authors would like to thank Stanley Snelson for guidance on continuation results for the hard potentials case.  CH was supported by NSF grants DMS-2617615, 2337666, and 2204615. LS was supported by NSF grants DMS-2054888 and DMS-2350263.}

\begin{document}

\begin{abstract}
    We establish a new {\em a priori} estimate on solutions to the space-inhomogeneous Landau and Boltzmann equations. As a consequence, we prove a new continuation criterion, based on a weighted $L^\infty$-norm, without requiring bounds on the hydrodynamic quantities. This complements existing conditional regularity results from a rather different perspective.

    Consequently, we show that the singularities present in the fluid equations are largely incompatible with the Boltzmann and Landau equations. More precisely, we largely rule out ``lifting a singularity'' from the 3D Euler equations to the physical range of kinetic equations, a widely expected mechanism for singularity formation. Under general considerations, this mechanism is essentially excluded for soft potentials, whereas for hard potentials the situation is more nuanced: one cannot produce blowup through the standard hydrodynamic ansatz using {\em known} imploding solutions to the Euler equations.
\end{abstract}

\maketitle

\section{Introduction}

\subsection{Setting}\label{ss:setting}
We are concerned with the spatially inhomogeneous Landau and Boltzmann equations, written
\be\label{eq:kinetic}
    (\partial_t + v\cdot\nabla_x) f
        = {\rm Q}(f,f).
\ee
Here, $\Q$ may be either the Landau or Boltzmann collision operator. A major open problem in the analysis of kinetic equations is to determine whether or not solutions to~\eqref{eq:kinetic} can blow up in finite time; that is, do solutions remain smooth or can they develop singularities?

Despite recent progress, this problem remains open for all of the usual nonlinear collision operators $\Q$, and even a clear heuristic is absent.
If singularities do occur, it is important to understand their qualitative and quantitative structure and the conditions under which they form.
From this point of view, a natural first step is to establish continuation criteria that rule out the formation of singularities when certain simple quantities remain bounded, thereby reducing the problem to a question of whether more tractable norms blow up. 
We recall the following statement by John Nash \cite{nash1958continuity}, originally in the context of fluid equations, but which applies equally well to kinetic equations:
\begin{center}
\includegraphics[width=0.78\textwidth]{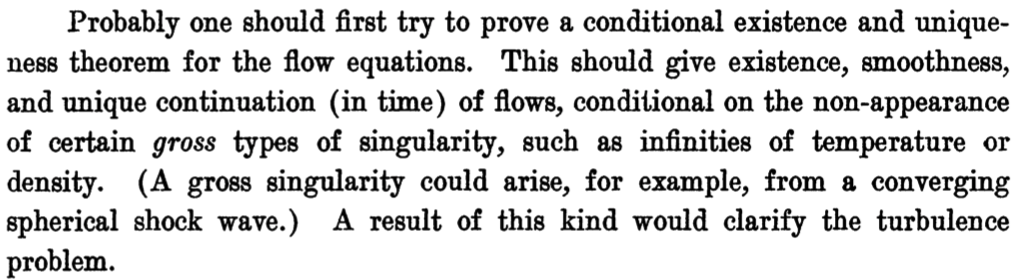}
\end{center}

In recent years, there has been substantial progress on conditional regularity for the Landau and non-cutoff Boltzmann equations. For moderately soft potentials, singularities cannot form as long as the mass and energy densities are bounded~\cite{HendersonSnelson,imbert2022global,imbert2020regularity,HendersonSnelsonTarfulea1,henderson2020self}; the only scenario not yet ruled out is concentration of mass or energy in space (see the discussion in \cite[Section 1.3]{imbert2020regularity}).

Heuristically, this suggests that any singularity formation at the kinetic level should already be visible in the associated fluid equations. Indeed, the hydrodynamic quantities associated with~\eqref{eq:kinetic} approximately solve the compressible Euler equations in an appropriate scaling limit (see, e.g.,~\cite{BardosGolseLevermore1,BardosGolseLevermore2,Golse_hydrodynamic}). However, the compressible Euler equations admit a wide range of singular behaviors. It is therefore natural to ask which of these singularities can manifest in the kinetic description, and how. 

For example, during shock formation, the mass and energy densities remain bounded while their gradients undergo blowup; existing continuation criteria therefore rule out shocks as a source of singularities at the kinetic level. More fundamentally, shock formation corresponds to the breakdown of a single-valued velocity field, whereas the kinetic description naturally allows multiple velocities to coexist at a given point in space-time. This provides a structural reason why shocks do not produce singularities at the kinetic level.

No analogous structural obstruction is known for implosions. Indeed, in light of a recent construction of singular solutions to the Landau equation in \cite{bedrossian2026finite}, any incompatibility with the kinetic dynamics must be more subtle. In this work, we prove a rather different continuation criterion (Theorem \ref{thm:main_propagation} below) and argue that it imposes severe structural obstructions on \emph{hydrodynamic} implosions, namely implosions arising from the connection between \eqref{eq:kinetic} and the Euler equations.

\medskip
\noindent\textbf{The hydrodynamic blowup scenario.}

\smallskip

A classical construction by Guderley~\cite{Guderley} demonstrated self-similar imploding solutions to the compressible Euler system based on a converging shock wave (see also~\cite{JangLiuSchrecker,JangLiuSchrecker2}). More recently, Merle, Rapha\"el, Rodnianski, and Szeftel~\cite{MerleRaphaelRodnianskiSzeftel} constructed smooth self-similar imploding solutions, relying on geometric focusing mechanisms present in the Euler equations. Subsequent constructions have extended the range of similarity exponents and clarified the essential features of these solutions \cite{BuckmasterCaoGomez,cao2025non,buckmaster2023smooth}.

A natural idea is to use such imploding Euler flows to construct a blowup scenario for the kinetic equation~\eqref{eq:kinetic}. This amounts
to constructing a solution whose macroscopic quantities remain close to those of the imploding Euler solution. In this regime, the microscopic velocity profile, to leading order, is the local Maxwellian corresponding to the Euler flow. In other words, the series expansion for the Boltzmann or Landau equation (as in Caflisch's work~\cite{Caflisch}) involves a leading order term explicitly determined by the underlying solution to the Euler equations. Smooth imploding solutions with nonzero density and pressure provide a leading order term that is smooth and non-vanishing, making them the most natural candidates for this approach.

This program has been recently carried out successfully to produce an implosion-type singularity for the Landau equation by Bedrossian et al.~\cite{bedrossian2026finite} for a ``very hard'' potential $\gamma > \sqrt{3}$, outside the physical range of power-law particle systems. 
A notable characteristic of this construction is that it produces solutions that remain globally bounded in $L^\infty$. The singularity forms as a result of \emph{tail fattening}; that is, the solution develops a large velocity tail around some point in space, causing blowup in the mass and energy densities.

In this paper, we identify a severe obstruction to this \emph{tail fattening} mechanism for physical potentials by propagating decay in velocity assuming only minimal weighted $L^\infty$-type control. Our result becomes stronger and more restrictive for softer potentials and strongest in the case of the Coulomb potential. In particular, an implosion singularity of the type constructed in~\cite{bedrossian2026finite} cannot arise in the physical range from any currently known smooth Euler implosion profile. The threshold $\gamma > \sqrt3$ appearing in~\cite{bedrossian2026finite}
therefore reflects a structural limitation of the hydrodynamic implosion mechanism, not merely a technical artifact of that construction.

These obstructions suggest that, in the physical range, any finite-time singularity would have to arise from a mechanism substantially different from the known \emph{hydrodynamic} implosion scenarios.
This leaves open the possibility of a \emph{genuinely kinetic} implosion mechanism, but rules out the most direct route from known Euler implosions.

\subsection{Main results}

\subsubsection{The mathematical setting}

We recall the standard formulations for the Landau and Boltzmann collision operators.
The Landau collision operator is given by
\be\label{eq:Landau}
    \QL(g,f)
    \coloneqq \nabla_v \cdot \int_{\R^d} \abs{v-w}^{2+\gamma} \Pi(v-w)\left[g(w)\nabla_v f(v) - f(v)\nabla_w g(w) \right] \dd w.
\ee
Here $\Pi_{ij}(z) = \delta_{ij} - \hat z_i \hat z_j$ is the projection matrix onto the plane orthogonal to $z$. The parameter $\gamma$ is in the range $\gamma \in [-d,1]$.  The case $\gamma = -3$, with $d=3$, corresponds to the Landau-Coulomb operator in three dimensions.  We consider $d \geq 2$.  In the special case $d=2$, we assume $\gamma \in (-2,1]$.  This is explained in greater detail in Section \ref{sec:Landau_prelim}.

The Boltzmann collision operator is given by
\be\label{eq:Boltzmann}
    \QB(g,f)
    	\coloneqq \int_{\R^d}\int_{\S^{d-1}} B(v-v_*, \sigma) \left[f(v')g(v_*')  - f(v) g(v_*)\right] \dd \sigma \dd v_*,
\ee
where $v'$ and $v_*'$ are defined through the relations
\be\label{e.v'_v_*'_sigma}
	v' = \frac{v + v_*}{2} + \frac{|v-v_*|}{2} \sigma
	\quad\text{ and }\quad
	v_*' = \frac{v + v_*}{2} - \frac{|v-v_*|}{2} \sigma.
\ee
The collision kernel $B\not\equiv0$ is a nonnegative function satisfying, for $\gamma > -d$,
\be\label{e.collision_kernel}
	B(w, \sigma) = |w|^\gamma \ b(\sin \theta/2)
	\qquad\text{ where } \cos \theta = \sigma \cdot \frac{w}{|w|}
	\quad\text{ and }\quad
    \int_{\S^{d-1}} |\sin(\theta/2)|^2 \ b(\sin \theta/2) \dd\sigma < \infty,
\ee
where $b(\sin\theta/2)$ is an arbitrary even function of the deviation angle $\theta$.
We stress that we {\em only} assume~\eqref{e.collision_kernel} for our main {\em a priori} estimates that we prove in this paper.  In particular, we consider {\em both} cutoff and non-cutoff Boltzmann collision kernels simultaneously.

\subsubsection{The a priori estimate}
The weighted spaces $L^\infty_m$ are defined in Section \ref{ss.notation} and contain bounded functions with pointwise decay like $|v|^{-m}$. The precise notion of solution and comments regarding it are contained in Section \ref{sec:preliminaries}. We are ready to state our main result:

\begin{theorem}\label{thm:main_propagation}
Let $\Q$ be either the Landau or Boltzmann collision operator; $m > m_0$, for $m_0$ sufficiently large depending on $\Q$;
and $f\colon[0,T]\times \Omega \times \R^d \to \R^+$ be a solution to~\eqref{eq:kinetic} with initial data $f_{\rm in} \in L^\infty_m$.
 Then,
\begin{equation}\label{e.a_priori}
    \|f(t)\|_{L^\infty_m}
    \le
        \norm{f_{\rm in}}_{L^\infty_m}
        \exp\left(
            C \int_0^t \norm{f(s)}_{L^\infty_{d+\gamma}} \dd s
            \right).
\end{equation}
Here, $C$ is a constant that depends on $m$ and $\Q$, and the spatial domain $\Omega$ is either $\T^d$ or $\R^d$.
\end{theorem}

Theorem \ref{thm:main_propagation} says that polynomial decay in velocity propagates as long as \(\|f(t)\|_{L^\infty_{d+\gamma}}\) is integrable in time.
In particular, if the initial data is rapidly decaying in velocity, then this decay persists for as long as the right hand side of~\eqref{e.a_priori} remains finite. Applying Theorem \ref{thm:main_propagation} with $m > \max(m_0,d + 2)$ controls the mass and energy densities:
\begin{equation}\label{eq:mass_energy}
    \rho(t,x) \coloneqq \int_{\R^d} f(t,x,v) \dd v
    \qquad\text{ and }\qquad
    \mathcal{E}(t,x) \coloneqq \int_{\R^d} \abs{v}^2 f(t,x,v) \dd v.
\end{equation}
Consequently, no implosions can occur while $\|f(t)\|_{L^\infty_{d+\gamma}}$ is integrable in time. Additionally, once the mass and energy densities are controlled, existing conditional regularity theory applies, and no singularity formation can occur in higher order norms (see Theorem \ref{thm:main_continuation} below). The subtle question remains: How restrictive is the condition on the $L^1_tL^\infty_{d+\gamma}$-norm?

The answer is counterintuitive and should be read with the hydrodynamic limit in mind.

\medskip
\noindent\textbf{Decay versus integrability.}

\smallskip

Contrary to prior continuation results, the strength of Theorem \ref{thm:main_propagation} lies in the decay requirements in $v$, \emph{not} in integrability requirements in $v$. The hydrodynamic quantities appearing in \eqref{eq:mass_energy} are $L^1$-type moments; if measured in terms of pointwise decay, controlling the energy density roughly requires decay of order $\abs{v}^{-d-2}$ as $\abs{v} \to +\infty$. By contrast, finiteness of the $L^\infty_{d + \gamma}$-norm imposes pointwise decay of order $\abs{v}^{-d - \gamma}$ as $\abs{v}\to +\infty$. Thus, as a decay condition, the $L^\infty_{d+\gamma}$-norm can be substantially weaker than control over the energy density, especially as $\gamma$ becomes more negative.

As discussed above, our unusual perspective here comes from the connection to the Euler equations. The lifting procedure connecting~\eqref{eq:kinetic} with the Euler equations yields bounded solutions. In our setting, the most important feature of a norm is in how much decay it requires as $|v| \to +\infty$.

To illustrate, we consider the Coulomb case $\gamma = -d = -3$. Theorem \ref{thm:main_propagation} implies that velocity decay of any order is propagated assuming only an \emph{unweighted} $L^\infty$-bound on $f$, which strikingly imposes no decay requirements on $f$. This exactly prohibits the \emph{tail fattening} singularity formation scenario discussed above.

As a further illustration, note that the collision operators in \eqref{eq:Landau} and \eqref{eq:Boltzmann} are nonlocal in $v$; to make sense of $Q(f,f)$ requires some decay of $f$. However, in the Coulomb case, our control quantity imposes no decay; the qualitative decay necessary to make sense of $Q$ is propagated from the initial data. In particular, one cannot prove Theorem \ref{thm:main_propagation} by viewing \eqref{eq:Landau} as an elliptic operator with bounded coefficients. More precise structure of the operator is needed.
Given the importance of decay to even make sense of the collision operator, it is significant and surprising that we can use such minimal decay on $f$ to propagate fast decay at $|v| \to +\infty$.

\medskip
\noindent\textbf{Idea of the proof.}

\smallskip

The above discussion hints at the fundamental difficulty in proving Theorem \ref{thm:main_propagation}: the nonlocality of the collision operator. Nonlocal effects can unavoidably destroy decay properties: For instance, $\Delta^{-1}u$ need not decay arbitrarily fast in $\R^d$,
even if $u$ is compactly supported. Similar convolutions occur in the Landau operator \eqref{eq:Landau} and even stronger nonlocal effects are present in the Boltzmann operator \eqref{eq:Boltzmann}.

The proof of Theorem \ref{thm:main_propagation} is a nonlinear barrier argument carried out in Section \ref{sec:propagation}, based on an intricate estimate of $Q(f,f)$ at any contact point with a polynomial barrier. The main observation is that at such a contact point, the potentially fatal nonlocal contributions are not arbitrary, but rather have a good sign once properly isolated. Accounting for all of the nonlocal contributions is the main difficulty in proving Theorem \ref{thm:main_propagation}. This is achieved by a decomposition for both the Landau and Boltzmann operators in Section \ref{sec:local_nonlocal_collisions}. After the worst nonlocal terms have a sign, the decay of the remaining terms can be computed by scaling.

In the Landau case, the proof gives the simple explicit threshold $m_0=2+d+\gamma$. The Boltzmann case is more complicated; although $m_0$
can in principle be extracted from our argument, its value depends strongly on structural assumptions on the angular kernel $b(\sin\theta/2)$.

\subsubsection{Alternative lower bound on blowup rates}

    The \emph{a priori} estimate contained in Theorem \ref{thm:main_propagation} can be interpreted as a lower bound on the blowup rate of the $L^\infty_{d + \gamma}$-norm during singularity formation.
    We are also able to produce an \emph{alternative} lower bound:
    \begin{corollary}\label{cor:lower_bound_blowup_rate}
       Let $\Q$ be either the Landau or Boltzmann collision operator. Suppose $f\colon[0,T)\times \Omega \times \R^d \to \R^+$ is a solution to~\eqref{eq:kinetic} with initial data $f_{\rm in} \in L^\infty_{m}$ for $m > m_0$. 
    Then, 
    \begin{equation}
        \mathrm{if} \quad\limsup_{t \to T^-} \norm{f(t)}_{L^\infty} = +\infty, \qquad \mathrm{then}\quad \norm{f(t)}_{L^\infty_m} \ge \frac{1}{C(T-t)}
\end{equation}
for all $t\in [0,T)$. 
Here, $C$ and $m_0$ are the same constants as in Theorem \ref{thm:main_propagation}, and the spatial domain $\Omega$ is either $\T^d$ or $\R^d$.
    \end{corollary}

    Corollary \ref{cor:lower_bound_blowup_rate} implies that sufficiently strong weighted $L^\infty$-norms cannot blow up more slowly than $(T-t)^{-1}$ during singularity formation at time $T$, for either Boltzmann or Landau; the power is determined solely by the quadratic structure of the collision operators. Because $(T-t)^{-1}$ is non-integrable, this is a strictly faster rate than that guaranteed by Theorem \ref{thm:main_propagation}. However, Theorem \ref{thm:main_propagation} applies to norms with much less stringent decay, i.e. $\norm{\cdot}_{L^\infty_{d+\gamma}}$ as opposed to $\norm{\cdot}_{L^\infty_{m_0}}$. 

    Corollary \ref{cor:lower_bound_blowup_rate} is clearest in the case of Coulomb potential ($-\gamma = d = 3$). In this case, \eqref{eq:Landau} takes the form 
    \begin{equation*}
        \partial_t f = \bar a_{ij} \partial_{ij} f + 8\pi f^2,
    \end{equation*}
    where $\bar{a}_{ij}$ is a diffusion coefficient.
    The unweighted norm $\norm{f(t)}_{L^\infty}$ is a subsolution to a Riccati equation; solving the ODE, the comparison principle implies $\norm{f(t)}_{L^\infty}$ cannot grow slower than exactly $[8\pi (T-t)]^{-1}$.
    For the other collision operators, Theorem \ref{thm:main_propagation} implies $\norm{f(t)}_{L^\infty_{m}}$ is a subsolution to the same Riccati equation with a different constant.
    
    Lastly, we are unaware of the above phrasing of Corollary \ref{cor:lower_bound_blowup_rate} anywhere in the literature. We remark that, for \emph{soft potentials} ($\gamma < 0$), it is essentially a consequence of \cite[Lemma 2.3]{HendersonSnelsonTarfulea} for Landau or \cite[Lemma 4.1]{henderson2025decay} for non-cutoff Boltzmann. For hard potentials, where moment estimates rarely close, Corollary \ref{cor:lower_bound_blowup_rate} is not a simple consequence of a prior result in the literature.

\subsubsection{The continuation criterion}

Theorem \ref{thm:main_propagation} is an a priori estimate that holds broadly. However, to turn it into a rigorous continuation theorem, one must choose a local theory for the collision operator at hand. This is awkward, especially for Boltzmann, because the available results vary substantially based on assumptions on initial data and between cutoff or non-cutoff and hard or soft potentials.
Rather than reprove or unify the fragmented local theory, we isolate an informal meta-theorem that arises from combining Theorem \ref{thm:main_propagation} with the existing regularity theory. Then, as one concrete example, we present  Theorem \ref{thm:main_continuation} which covers the usual collision operators with soft potentials.
\begin{quote}
\textbf{Informal Continuation Principle.} \em
Fix any choice of Landau or Boltzmann collision operator $\Q$. Suppose there exists a local-in-time existence or continuation result for~\eqref{eq:kinetic} with $\Q$ that requires only finiteness of a polynomially weighted Sobolev norm. Then the solution $f$ is classical and can be continued beyond $[0,T]$ if
\begin{equation}\label{e.continuation}
    \int_0^T \|f(t)\|_{L^\infty_{d+\gamma}} \dd t < \infty.
\end{equation}
\end{quote}

The point is simple. While the details vary, existing continuation results for the Landau and non-cutoff Boltzmann equations typically require control of sufficiently many derivatives and sufficiently many polynomial moments. Theorem \ref{thm:main_propagation} supplies the polynomial moments.
Once these moments are available, the existing regularity theory yields the remaining bounds needed to continue the solution. Thus the possible breakdown of a classical solution is reduced to the blowup of this one weighted $L^\infty$-norm.

We now give a concrete formulation for soft potentials  in three dimensions, where the required existence theory for polynomially weighted initial data is available. In the Boltzmann case, to obtain regularization estimates, we introduce the standard ``non-cutoff'' condition:
\begin{equation}\label{e.non-cutoff}
	b(\sin\theta/2) \approx \left|\sin \frac\theta2\right|^{ - (d-1) - 2s} \qquad \text{with} \quad s\in(0,1),
\end{equation}
which produces the smoothing effect used in the continuation result~\cite{imbert2020regularity}.
For inverse-power-law interactions in three dimensions, the parameters $s$ and $\gamma$ are related by $\gamma = 1 - 4s$, so the physical range corresponds to $-3\le \gamma \le 1$. The limiting case $s=1$ corresponds formally to the Landau collision operator.

\begin{theorem}\label{thm:main_continuation}
Let $\Q$ be either the Landau or the non-cutoff Boltzmann collision operator (recall~\eqref{e.non-cutoff}).  Assume that  $\gamma \in [-3,0)$ and, in the Boltzmann case, that $s\in (0,1)$.
Suppose $f\colon[0,T^*)\times \Omega \times \R^3 \to \R^+$ is a solution to~\eqref{eq:kinetic} with continuous initial data $f_{\rm in} \in L^\infty_m$ for $m$ sufficiently large.  Assume that $[0,T^*)$ is the maximal interval of existence. Then, either $T^* = \infty$ or
    \begin{equation*}
        \int_0^{T^*} \|f(t)\|_{L^\infty_{3+\gamma}} \dd t = \infty.
    \end{equation*}
Here, the spatial domain $\Omega$ is either $\T^3$ or $\R^3$.
\end{theorem}

Theorem \ref{thm:main_propagation} already prevents the implosion mechanism by itself: while the weighted $L^\infty$-norm remains integrable, velocity decay is propagated and mass or energy cannot concentrate through tail fattening. Theorem \ref{thm:main_continuation} unifies this estimate with the available local theory, simultaneously ensuring that no other singularity mechanism already covered by the existing continuation criteria can occur.

For the Landau equation, there is currently no local well-posedness for the hard potentials case $\gamma \in [0,1]$ using only polynomial weights.  We refer the reader to~\cite{Chaturvedi,SnelsonTaylor} for work in this direction but with (stretched) exponential weights. For the Boltzmann equation, a result was recently announced in~\cite{li2026local_hardpotentials}.

\begin{remark}
    Notice that the \emph{a priori} estimate in Theorem \ref{thm:main_propagation} yields an unconditional short time bound for $m > \max(m_0,d+\gamma)$ for any $\gamma$. Theorem \ref{thm:main_propagation} is stable under approximation, e.g. under truncation of the kernel in $\Q$ or added diffusion. Consequently, one could likely use a version of Theorem \ref{thm:main_propagation} to sidestep the ``moment closure problem'' for hard potentials, which leads to involved commutator estimates and towers of norms in ~\cite{Chaturvedi,li2026local_hardpotentials}. We do not investigate this direction further here.
\end{remark}

\subsection{Organization}
We begin this work in Section \ref{sec:Euler} with a discussion of the implications of Theorem \ref{thm:main_propagation} and Theorem \ref{thm:main_continuation} for using singularity formation in the Euler equations to construct blowup to \eqref{eq:kinetic}.
Afterwards, we recall some known results in Section \ref{sec:preliminaries}.  In Section \ref{sec:propagation}, we show how to deduce Theorem \ref{thm:main_propagation} from estimates on the collision operator (Lemma \ref{lem:main_lemma}). This is the crux of the work, and it is, in turn, broken down into lemmas regarding the contributions of collisions involving small and large velocities in Section \ref{sec:local_nonlocal_collisions}. To this point, the work is quite general, using nothing particular about the collision operators.  We specialize to the Landau case in Section \ref{sec:Landau_a_priori} and the Boltzmann case in Section \ref{sec:Boltzmann_a_priori}.  Finally, we prove Corollary \ref{cor:lower_bound_blowup_rate} and Theorem \ref{thm:main_continuation} in Section \ref{sec:continuation}.

\subsection{Notation and terminology}\label{ss.notation}
For any vector $v \in \R^d$, we use the following standard notation:
\begin{equation}
    \hat v = \frac{v}{|v|}
    \qquad\text{ and }\qquad
    \brak{v} = \sqrt{1 + |v|^2}.
\end{equation}
Weighted $L^p$-spaces are defined using the bracket via the norm
\begin{equation*}
    \norm{f}_{L^p_m}^p
    \coloneqq \norm{\brak{\cdot}^mf}_{L^p}^p
    = \int_{\R^d} \brak{v}^{mp} f^p \dd v,
\end{equation*}
with the standard modification when $p = \infty$. Whenever the domain of integration is omitted from an integral, the integral should be understood to be taken over the entire domain, i.e. $\Omega = \mathbb{T}^d$ or $\R^d$ if integrated with respect to $\dd x$ or $\R^d$ if integrated with respect to $\dd v$.

We use the notation $A \lesssim B$ if $A \leq C B$ for some universal constant $C$ that we allow to depend only on $d$, $\gamma$, $b$, and $m$.

\section{Relationship to state-of-the-art Euler blowup scenarios}\label{sec:Euler}

In this section, we show that the continuation criterion developed above severely restricts the pathway to blowup for \eqref{eq:kinetic} via \emph{hydrodynamic implosion}.\footnote{Recall that by a \emph{hydrodynamic implosion}, we mean precisely a singular solution to \eqref{eq:kinetic} where the leading order behavior is dictated by a local Maxwellian where the associated parameters correspond to an implosion singularity for the Euler equations.}
Using the {\em a priori} estimate in Theorem \ref{thm:main_propagation} and some basic theory of the Euler equations, we argue that:
\begin{itemize}

    \item ($\gamma = -3$) There are no hydrodynamic implosions.

    \smallskip

    \item ($\gamma < -1/3$) There are no \emph{radial, approximately self-similar} hydrodynamic implosions compatible with \emph{local} conservation of mass and energy.

    \smallskip

    \item ($\gamma < \sqrt{3}$) There are no \emph{known} implosions to the Euler equations that can be used to produce a hydrodynamic implosion. 

\end{itemize}
We emphasize that Theorem \ref{thm:main_propagation} is sufficiently sharp to identify the threshold $\gamma = \sqrt{3}$ appearing in \cite{bedrossian2026finite} as a fundamental, not technical, barrier for \emph{known implosions}.

First, we review preliminaries about the compressible Euler equations, culminating in \eqref{eq:prefactor_bound}, a crucial consequence of the transport of the specific entropy. Second, we introduce the hydrodynamic limit connecting the kinetic and macroscopic levels leading to the ansatz~\eqref{eq:ansatz} for potentially singular solutions to the Landau and Boltzmann equations. Third, we apply Theorem \ref{thm:main_propagation} to this ansatz, resulting in a case-by-case analysis.

\subsection{The compressible Euler equations}

The full, temperature dependent 3D Euler equations can be written as the following system of nonlinear conservation laws:
\begin{equation}\label{eq:Euler}
\begin{aligned}
    \partial_t \rho + \nabla_x \cdot (\rho u) &= 0\\
    \partial_t (\rho u) + \nabla_x \cdot (\rho u\otimes u +  pI) &= 0\\
    \partial_t (\rho E) + \nabla_x \cdot (\rho E u + pu) &= 0,
\end{aligned}
\end{equation}
where $\rho$ is the mass density, $u$ is the macroscopic velocity field, $p$ is the pressure, and $E$ is the total energy density. For a compressible fluid, additional constitutive relations are necessary to close the system. For a monatomic ideal gas, these relations take the following form in terms of $\theta$, the temperature: \footnote{In the fluid mechanics literature, it is substantially more common to work with the internal energy $e$ in place of the temperature. For a monatomic gas, $e = \frac{3}{2}\theta$ and \eqref{eq:pressure_law} is the familiar equation of state for a polytropic gas with adiabatic exponent $5/3$. Additionally, we work in units so that the Boltzmann constant is $1$.}
\begin{equation}\label{eq:pressure_law}
    p = \rho \theta \qquad \text{and} \qquad E = \frac{3\theta}{2} + \frac{\abs{u}^2}{2}.
\end{equation}
Additionally, we introduce the specific entropy $\bar S$:
\begin{equation}
    \bar S = \log\left(\frac{2\rho^{2/3}}{3\theta}\right).
\end{equation}
For admissible solutions of \eqref{eq:Euler}, $\bar S$ is a subsolution to both a divergence-form and pure transport equation:
\begin{equation}\label{eq:entropy}
    \partial_t \left(\rho \bar S\right) + \nabla_x \cdot \left( \rho u \bar S\right) \le 0 \qquad \text{and} \qquad \partial_t \bar S  + u \cdot \nabla_x \bar S \le 0.
\end{equation}
Equality holds in smooth regions of fluid flow, but the inequality can become strict when admissible discontinuities (e.g. shock waves) are present.
A simple application of the maximum principle implies
\begin{equation}\label{eq:maximum_principle_specific_entropy}
    \sup_{x \in \R^3} \bar S(t,x) \le \sup_{x \in \R^3} \bar S(0,x).
\end{equation}
As a consequence, we conclude that for $(\rho,u,\theta)$ any entropy admissible solution to \eqref{eq:Euler} on $[0,T_*)\times \R^3$,
\begin{equation}\label{eq:prefactor_bound}
    \rho^{2/3} \le C \theta, 
\end{equation}
where the constant $C$ depends only on the right hand side of~\eqref{eq:maximum_principle_specific_entropy}. This inequality is crucial in the sequel.

\begin{remark}
    The above computations leading to \eqref{eq:prefactor_bound} are rigorous provided $(\rho,u,\theta)$ is sufficiently regular on $[0,T_*)\times \R^3$ and avoids both $\rho = 0$ and $\theta = 0$. However, \eqref{eq:prefactor_bound} can also be extended to a substantially larger class of weak solutions to \eqref{eq:Euler} including solutions that are merely bounded before time $T_*$, but may contain vacuum or entropy-admissible shocks. The key point is \eqref{eq:entropy} persists as an inequality for admissible solutions. This can be justified on physical grounds using the second law of thermodynamics and is valid for vanishing viscosity limits of the Navier-Stokes equations. More importantly for us, the inequality is also valid for limits of kinetic equations as a consequence of the decay of the Boltzmann entropy (see \cite{BardosGolseLevermore1}).

    On the other hand,~\eqref{eq:prefactor_bound} can fail if the initial data contains ``cold gas'' regions, i.e. non-vacuum regions of zero temperature---equivalently zero pressure. However, \eqref{eq:prefactor_bound} precludes the formation of such regions.
\end{remark}

\subsection{The strongly collisional limit}\label{ss:strongly_collisional}

The usual Boltzmann and Landau kinetic equations are related to the system~\eqref{eq:Euler} via the strongly collisional limit. Work in this direction, dating back to Caflisch~\cite{Caflisch} (see also the work of Bardos, Golse, and Levermore~\cite{BardosGolseLevermore1,BardosGolseLevermore2}), investigates the limit of
\begin{equation}\label{eq:kinetic_collisional}
     (\partial_t + v \cdot \nabla_x) f_{\eps} = \frac{1}{\eps}\Q(f_{\eps},f_\eps),
\end{equation}
as $\eps\to0$. 
Formally one sees that the limit must be an element of the kernel of $\Q$, which, under very general assumptions on $\Q$, consists of {\em local Maxwellians}.  These are functions of the form:
\begin{equation}\label{eq:local_Maxwellian}
    M(t,x,v) \coloneqq \frac{\rho(t,x)}{(2\pi \theta(t,x))^{3/2}} \exp\left(-\frac{\abs{v - u(t,x)}^2}{2\theta(t,x)}\right).
\end{equation}
Observe that $M$ depends on $t$ and $x$ only through the mass density $\rho$, the macroscopic velocity $u$, and the temperature $\theta$.
Additionally, the triple $(\rho,u,\theta)$ can be recovered as moments of $M$ via the relations,
\begin{align*}
    \rho(t,x) &= \int_{\R^3} f(t,x,v) \dd v\\
    \rho u(t,x) &= \int_{\R^3} v f(t,x,v) \dd v, \\
    \rho \theta(t,x) &= \frac 13 \int_{\R^3} \abs{v-u(t,x)}^2 f(t,x,v) \dd v.
\end{align*}
A simple consequence of the kinetic equation \eqref{eq:kinetic_collisional} and elementary computations is that $(\rho,u,\theta)$ solve the compressible Euler equations~\eqref{eq:Euler} with the constitutive relations \eqref{eq:pressure_law}.

\begin{remark}
    Under mild assumptions on the collision operator $\Q(f)$, the kinetic equation \eqref{eq:kinetic_collisional} converges in the strongly collisional regime to the compressible Euler equations with adiabatic exponent $5/3$. In particular, the 3D Euler equations with other pressure laws---even another adiabatic exponent---are \emph{not} observed as limits of the standard Landau and Boltzmann equations.
\end{remark}

\begin{remark}
    The form of the local Maxwellian is degenerate in the ``cold gas'' regions where $\theta = 0$. Instead, one should replace $M(t,x,v)$ by the measure $\rho(t,x)\delta_{\set{v = u(t,x)}}$.
\end{remark}

The study of singularity formation for the compressible Euler equations is considerably more mature than the corresponding area for kinetic equations. In particular, it is known that \eqref{eq:Euler} admits a wide range of self-similar imploding solutions~\cite{Guderley,MerleRaphaelRodnianskiSzeftel,MerleRaphaelRodnianskiSzeftel2,buckmaster2023smooth,BuckmasterCaoGomez,shao2025blow,cao2025non,JangLiuSchrecker,JangLiuSchrecker2,JenssenTsikkou1}.
Given the connection between kinetic equations and the 3D Euler equations outlined in Section \ref{ss:strongly_collisional}, one might hope to ``lift a singularity'' from the Euler equations to the kinetic level. 
More precisely, we make the ansatz
\begin{equation}\label{eq:ansatz}
    f(t,x,v) = \underbrace{\frac{\rho(t,x)}{(2\pi\theta(t,x))^{3/2}} \exp\left(-\frac{\abs{v-u(t,x)}^2}{2\theta(t,x)}\right)}_{M(t,x,v)} \quad + \quad \text{Smaller Perturbation},
\end{equation}
where $(\rho,u,\theta)$ is a singular solution to \eqref{eq:Euler} and $M(t,x,v)$ is the corresponding local Maxwellian. We now apply Theorem \ref{thm:main_propagation} to solutions $f(t,x,v)$ of the Landau and Boltzmann equations of the form \eqref{eq:ansatz}.

\subsection{Hydrodynamic implosions: the Landau-Coulomb case}
Note that since $(\rho,u,\theta)$ is a solution to the Euler equations \eqref{eq:Euler}, the prefactor in front of the local Maxwellian remains bounded by \eqref{eq:prefactor_bound}. Consequently, the local Maxwellian remains bounded uniformly in time:
\begin{equation}
    \norm{M(t)}_{L^\infty_{x,v}} \le \sup_{x\in\R^3} \frac{\rho(t,x)}{(2\pi \theta(t,x))^{3/2}}\lesssim 1.
\end{equation}
Thus, as long as the perturbation in \eqref{eq:ansatz} is lower order, or $L^1(0,T;L^\infty)$, Theorem \ref{thm:main_continuation} implies no singularity formation can occur at the kinetic level.

Let us emphasize something notable: there are virtually no assumptions on the form of singularity to the 3D Euler equations above. In particular, we do not assume that the solution is isentropic, radial, or even self-similar. However, in order to be compatible with the ansatz~\eqref{eq:ansatz}, we must make the mild assumption that the temperature is non-zero in all non-vacuum regions, which is propagated by entropy-admissible solutions to the 3D Euler equations. Consequently, the approach of ``lifting a singularity'' from the Euler equations \emph{using the ansatz}~\eqref{eq:ansatz} seems entirely doomed for Landau-Coulomb, since the error must be leading order. Finding a more singular solution to the Euler equations can never directly resolve this issue.

\subsection{Hydrodynamic implosions: other collision operators}
For other Boltzmann or Landau collision operators with $\gamma > -3$, to apply Theorem \ref{thm:main_propagation}, we must compute a weighted norm of the local Maxwellian:
\begin{equation}
    \norm{M(t)}_{L^\infty_{3+\gamma}} \le C+ \sup_{(x,v)} \left[\abs{v}^{3+\gamma} \exp\left(-\frac{\abs{v-u(t,x)}^2}{2\theta(t,x)}\right)\right] \lesssim 1 + \sup_{x} \theta^{\frac{3 + \gamma}{2}} + \sup_x \abs{u}^{3+\gamma}.
\end{equation}
Again, assuming that the perturbation in \eqref{eq:ansatz} is lower order, Theorem \ref{thm:main_propagation} implies no blowup occurs at time $T_*$ at the kinetic level provided
\begin{equation}\label{eq:blowup_rate}
    \int_0^{T_*} \left(\norm{u(t)}_{L^\infty}^{3+\gamma} + \norm{\theta(t)}_{L^\infty}^{\frac{3 + \gamma}{2}}\right) \dd t  < \infty.
\end{equation}
The minimum blowup rate imposed by \eqref{eq:blowup_rate} is \emph{a necessary condition} for \eqref{eq:ansatz} to viably produce a singular solution at the kinetic level. We now study blowup profiles to the Euler equations and assess when \eqref{eq:blowup_rate} is satisfied.
Each of the known implosions to \eqref{eq:Euler} is either approximately self-similar or globally self-similar, which are, up to leading order, of the form
\begin{equation} \label{eq:self_similar_Euler}
\begin{aligned}
    \rho(t,x) &= (T_*-t)^{\frac\kappa\lambda} R\left(\frac{|x|}{(T_*-t)^{1/\lambda}}\right)\\
    u(t,x) &= (T_*-t)^{\frac 1 \lambda - 1} U\left(\frac{\abs{x}}{(T_*-t)^{1/\lambda}}\right),\\
    \theta(t,x) &= (T_*-t)^{2\left(\frac 1\lambda - 1\right)} \Theta\left(\frac{\abs{x}}{(T_*-t)^{1/\lambda}}\right),
\end{aligned}
\end{equation}
where $\kappa, \lambda$ are two similarity exponents and the self-similar profiles $R$, $U$, and $\Theta$ are merely bounded.
Additionally, by symmetry we assume in \eqref{eq:self_similar_Euler} that an implosion occurs at time $t = T_*$ and location $x = 0$. Note that $\rho$, and consequently $\kappa$, is irrelevant for verification of \eqref{eq:blowup_rate}. Indeed, substituting \eqref{eq:self_similar_Euler} into \eqref{eq:blowup_rate}, we find a necessary condition on the similarity exponent $\lambda$:
\begin{equation}\label{eq:condition_self_similar_euler}
    \int_0^{T_*} \left(\norm{u(t)}_{L^\infty}^{3 +\gamma} + \norm{\theta(t)}_{L^\infty}^{\frac{3+\gamma}{2}}\right) \dd t = +\infty \qquad \text{only if }\qquad (3 + \gamma)\left(\frac{1}{\lambda} - 1\right) \le -1.
\end{equation}

\noindent\textbf{Observation 1}: Larger values of $\lambda$, i.e. more singular solutions to the Euler equations, can be used for a larger range of collision operators indexed by $\gamma$. However, there is an inherent limit: taking $\lambda \to \infty$ in \eqref{eq:condition_self_similar_euler}, we see that the ansatz \eqref{eq:ansatz} with a self-similar implosion of the form \eqref{eq:self_similar_Euler} can only be expected to produce blowup for $\gamma > -2$.

\medskip

\noindent\textbf{Observation 2}: There is a better upper bound on $\lambda$. The ansatz \eqref{eq:ansatz}, \eqref{eq:self_similar_Euler} ought to be compatible with the usual \emph{a priori} estimates for kinetic equations, namely conservation of mass, momentum, and energy, which requires:
\begin{equation}
    \sup_{0 < t < T_*} \iint (1 + \abs{v}^2)f(t,x,v) \dd v \dd x  < \infty.
\end{equation}
Since the local Maxwellian $M(t,x,v)$ is leading order near the implosion, the singular solution to Euler must have finite total mass and energy in any ball surrounding the implosion point. Combined with the transport of the specific entropy \eqref{eq:prefactor_bound}, it is known that this minimal condition for physicality severely restricts the range of $\kappa$ and $\lambda$ (see \cite[Section 4]{Jenssen}):
\begin{equation}
    -3 < \kappa \le -3(\lambda -1) \qquad \text{and} \qquad 1 < \lambda < \frac{5 + \kappa}{2}.
\end{equation}
Consequently, decay of entropy and conservation of mass and energy force $1 < \lambda < 8/5$. Substituting this bound into \eqref{eq:blowup_rate}, we find that the ansatz \eqref{eq:ansatz} with a self-similar implosion of the form \eqref{eq:self_similar_Euler} can only be expected to produce blowup for $\gamma > -1/3$.

\begin{remark}
    The \emph{a priori} estimates of Serre using his theory of compensated integrability \cite{serre} seem to produce the weaker bound on the similarity exponent $\lambda < 9/5$. At present, we are unaware of other {\em a priori} estimates for the Euler equations that might further limit the similarity exponent $\lambda$.
\end{remark}

Consequently, for soft potentials with $\gamma < -1/3$, the approach of~\eqref{eq:ansatz} appears doomed as well. However, there is a caveat: if one finds unbounded self-similar profiles $R$, $\Theta$, and $U$ or constructs a finite time singularity for the Euler equations that is not approximately self-similar \eqref{eq:self_similar_Euler}, and this solution is compatible with the lifting procedure~\eqref{eq:ansatz}, the above approach is possible.

\subsection{Applications to known blowup profiles from Euler}

Let us analyze the blowup rates of the known implosion scenarios for the compressible Euler equations.  We proceed in order of decreasing regularity.

\begin{itemize}
    \item\textbf{Smooth implosions:} In light of the discussion above, the candidate implosion profiles with the fewest obvious obstructions are the smooth blowup profiles in \cite{buckmaster2023smooth,shao2025blow}. For these implosions, the achievable self-similarity exponents are $1 < \lambda < 3-\sqrt 3 \approx 1.268$. In that case, Theorem \ref{thm:main_propagation} prevents the construction of a singular solution of the form \eqref{eq:ansatz} whenever $\gamma < \sqrt 3$. This contains the entire physical range, and it exactly complements the construction in~\cite{bedrossian2026finite}, which requires $\gamma > \sqrt 3$.

    \medskip

    \item\textbf{Finite regularity implosions:}
    There is a significantly simpler construction of merely continuous solutions in \cite{JenssenTsikkou1} that additionally continues the solution as an outgoing reflected shock wave. The mechanism for blowup is the self-similar focusing of a continuous spherically symmetric wave similar to \cite{MerleRaphaelRodnianskiSzeftel,MerleRaphaelRodnianskiSzeftel2}. Explicitly computing the similarity exponent in \cite[Section 4]{JenssenTsikkou1}, once again one finds $1 < \lambda < 3-\sqrt{3}$.  This corresponds to the same range as above.

    \medskip

    \item\textbf{Collapsing cavity:} Another family of self-similar implosions were first described by Hunter in \cite{Hunter} to model a collapsing cavity (a comprehensive study via a modern similarity approach can be found in \cite{Lazarus}). In this scenario, a symmetric vacuum region is invaded by a converging rarefaction wave. The wave accelerates and collapses to a point (resp. axis) in finite time. All of the hydrodynamic quantities diverge at the time of collapse and the range for the similarity exponents are $\kappa =  -3(\lambda - 1)$ and $1 <\lambda < \frac{15-5\sqrt{2}}{7} \approx 1.133$ (for cylindrical symmetry) or $1 < \lambda < 3-\sqrt{3}$ (for spherical symmetry). Once again, this corresponds to the same range of similarity exponents.

    \medskip

    \item\textbf{Converging shock wave implosions:} In a final family of self-similar implosions---first described by Guderley in \cite{Guderley}---a region of cold gas (referred to as the quiescent region) is invaded by a symmetric shock wave. The discontinuity accelerates and collapses to a point (resp. axis) in finite time. In this scenario, the velocity field and temperature diverge at the time of collapse, whereas the density remains finite; in the notation above, $\kappa = 0$.
    It is known that the cold gas assumption is necessary to obtain an exact self-similar solution, and that in this case, there is a unique self-similarity exponent $\lambda \approx 1.45$ (spherical symmetry) or $\lambda \approx 1.226$ (cylindrical symmetry)~\cite[Table 6.2]{Lazarus}.
    Unlike the smooth implosions discussed above, the Guderley solution is fundamentally non-isentropic, and formally \eqref{eq:blowup_rate} indicates it blows up fast enough to leave open singularity formation at the kinetic level with $\gamma$ in the physical range.

    The exact Guderley scenario can form from smooth data \cite{CialdeaShkollerVicol}, but crucially only from data that has zero temperature on an open set. Due to the entropy inequality, such cold-gas regions cannot form dynamically from data with strictly positive temperature. The \emph{exact Guderley scenario} is therefore incompatible with the lifting procedure \eqref{eq:ansatz}; the local Maxwellian must be replaced by a Dirac mass in a cold gas region.

    However, we cannot rigorously rule out the construction of a smooth, positive-pressure approximation to the Guderley solution. Such a hypothetical scenario would face severe technical difficulties at both the hydrodynamic and kinetic level. Recently, in \cite{chen2026smoothstableeulerimplosions}, Chen, Shkoller, and Vicol construct smooth, rapid implosions---meaning faster than those of \cite{MerleRaphaelRodnianskiSzeftel,MerleRaphaelRodnianskiSzeftel2}---while removing much of the most singular cold-gas structure, replacing the zero-temperature quiescent region by a single point of vanishing temperature. However, this still falls short of the strictly positive-pressure regime required for a nonsingular local Maxwellian ansatz. The remaining difficulties are not merely cosmetic: positive temperature near the implosion produces a counterpressure resisting the converging flow. This resistance should intensify as the gas compresses, and could substantially slow down the implosion.
    Whether there exists a rapid 3D Euler implosion that has strictly positive pressure and remains compatible with the ansatz \eqref{eq:ansatz} is an interesting open question.

\end{itemize}

\section{Preliminaries}\label{sec:preliminaries}

In this section, we collect some results from the literature that are relevant for our work and discuss the notion of solution used in this paper.  This includes an overview of known techniques and results in order to explain the assumptions we make in our notion of solution.

\subsection{Rewriting the collision operators}

In the introduction, we stated the standard form of the collision operators that is most physically relevant. Here we state two other decompositions that are useful.

\subsubsection{Another form of the Landau collision operator}\label{sec:Landau_prelim}

In this case, the collisional form of the operator $\Q$ was defined in \eqref{eq:Landau}. 
There are several equivalent formulations of $\QL$ that emphasize particular aspects of the operator. We will mostly use the non-divergence elliptic form of $\QL$ obtained by expanding the divergence in \eqref{eq:Landau}:
\be\label{eq:Landau_alternative}
    \QL(f,f)
    	= \bar a_{ij} \partial_{ij} f + \bar c[f] f,
\ee
where
\begin{equation}\label{eq:defn_coeff}
\begin{aligned}
    \bar a_{ij}[f](v)
        &= \int_{\R^d} \abs{v-w}^{2+\gamma} \Pi_{ij}(v-w) f(w) \dd w
    \qquad\text{ and }\qquad
    \\
    \bar c[f](v) &= -\partial_{ij} \bar a_{ij}
        = \begin{cases} (d-1) (d+\gamma) \int_{\R^d} \abs{v-w}^\gamma f(w) \dd w &\text{if } \gamma > -d
    \\
    (d-1)\abs{\mathbb{S}^{d-1}} f(v)
        &\text{if }\gamma = -d.
    \end{cases}
\end{aligned}
\end{equation}
The case $\gamma = -d$ is exceptional because, formally, two derivatives of the kernel is a Dirac mass, making $\bar{c}$ a local operator. This computation fails for $\gamma = -d$ in $d=2$, which is why we omit this case.

\subsubsection{Another form of the Boltzmann collision operator}\label{sss:Carleman}
There is an analogous and equally useful alternative formulation of $\QB$, which more precisely identifies the ``non-divergence, nonlocal elliptic'' structure in $\QB$. In the expression~\eqref{eq:Boltzmann}, we may add and subtract $f(v)f(v_*')$ to find
\begin{equation}
    \QB(f, f)(v)
    = \Qs(f, f)(v) + \Qns(f, f)(v).
\end{equation}
The first term has a subscript ``s'' to indicate ``singular,'' as it contains the non-integrable kernel when $\QB$ is non-cutoff. Our assumptions on $b$ are sufficiently general that $\Qs$ need not actually be singular; however, we keep the terminology for consistency with previous works. It can be expressed in spherical or Carleman coordinates with any of the following equivalent expressions. 
\begin{equation}\label{e.Qs}
\begin{split}
    \Qs(f, f)(v)
    &\coloneqq \int_{\R^d}\int_{\S^{d-1}} \left( f(v') - f(v)\right) f(v_*') B(v-v_*, \sigma) \dd \sigma \dd v_* \\
    &= \int_{\R^d} \left( f(v') - f(v)\right) \int_{v'_*-v \perp v'-v} f(v_*') B_1(|v-v'_*|,|v-v'|)  \dd v'_* \dd v' \\
    &= \int_{\R^d} f(v'_*) \int_{v'-v \perp v'_\ast-v} (f(v')-f(v)) B_2(|v-v'_*|,|v-v'|)  \dd v' \dd v'_*.
\end{split}
\end{equation}
In the above formulas, the variables $v$, $v'$, $v_*$, $v_*'$, $\theta$, and $\sigma$ are given by ~\eqref{e.v'_v_*'_sigma} and $B$ is given by~\eqref{e.collision_kernel}.
The kernels in Carleman coordinates take the following form (see \cite[Appendix A]{Silvestre_new_regularization}):
\begin{align*}
    B_1(|v-v'_*|,|v-v'|) &\coloneqq \frac{2^{d-1} B(|v-v_*|,\sigma)}{|v-v'| \, |v-v_*|^{d-2}} = \frac{2^{d-1} |v-v_*|^{\gamma+2-d} }{|v-v'|}b(\sin \theta/2),\\
    B_2(|v-v'_*|,|v-v'|) &\coloneqq \frac{2^{d-1} B(|v-v_*|,\sigma)}{|v-v'_*| \, |v-v_*|^{d-2}} = \frac{2^{d-1} |v-v_*|^{\gamma+2-d} }{|v-v'_*|} b(\sin \theta/2).
\end{align*}
In these coordinates, the diameter of the collision sphere $r$ and the deviation angle $\theta$ are 
\begin{equation*}
    \sin(\theta/2) = \frac{\abs{v-v'}}{\abs{v-v_*}} \qquad \text{and} \qquad r = \abs{v-v_*} = \abs{v'-v'_*}.
\end{equation*}
Additionally, the integrability condition \eqref{e.collision_kernel} for $b(\sin \theta/2)$ is equivalent to
\begin{equation}\label{eq:B_2 condition}
\int_{v'-v \perp v'_*-v} \frac{|v'-v|^2\abs{v-v'_*}}{r^{3+\gamma}} B_2(|v-v'_*|,|v-v'|) \, \dd v' \lesssim 1.
\end{equation}

The other term, which has a subscript ``ns'' to indicate ``nonsingular'' as the (potential) singularity in $b$ vanishes in its derivation, is given by
\begin{equation}\label{e.Qns}
\begin{aligned}
    \Qns(f, f)(v) &\coloneqq f(v)\int_{\R^d}\int_{\S^{d-1}} (f(v_*') - f(v_*)) B(v-v_*,\sigma) \dd \sigma\dd v_*\\
    &= C_b f(v) \int_{\R^d} f(v-w) |w|^\gamma \dd w.
\end{aligned}
\end{equation}
where $C_b$ is a constant depending on the cross section $b$ and $\gamma$ (see ~\cite[Lemma~5.2]{Silvestre_new_regularization} for a derivation).

\subsection{Known local well-posedness and continuation criteria}\label{sss:known_continuation}

We state some known continuation criteria that we are able to bootstrap to obtain Theorem \ref{thm:main_continuation} using our {\em a priori} estimate Theorem \ref{thm:main_propagation}.

\subsubsection{The Landau equation}

We recall a local-in-time well-posedness result for soft potentials in three dimensions from \cite{HendersonSnelsonTarfulea}.

\begin{theorem} \label{t:local-well-posedness-Landau-soft}
    Assume $\gamma \in [-3,0)$ and $\Q=\QL$ as in \eqref{eq:Landau}. Let $m > \max(5, 15/(5 + \gamma))$. Assume $f_{\rm in}$ is a continuous nonnegative function such that $f_{\rm in} \in L^\infty_m(\R^6)$.
\begin{enumerate}

    \item There exists $T>0$ and a classical solution $f : [0,T]\times \Omega \times \R^3 \to [0,\infty)$. The function $f$ is strictly positive for $t>0$. The function $f$ is continuous up to $t = 0$ and $f(0,x,v) = f_{\rm in}(x,v)$.

    \item For any $k$, there is $m_k$ such that, if $m > m_k$, then $D_{t,x,v}^k f$ is continuous and bounded.

    \item Fix any $m_0 > \max(5, 15/(5+\gamma))$.  The solution can be extended for as long as $\|f(t)\|_{L^\infty_{m_0}} < +\infty$ with finite $L^\infty_m$-norm.

\end{enumerate}
\end{theorem}

Theorem \ref{t:local-well-posedness-Landau-soft}.(iii) is a simple example of  a continuation criterion for the Landau equation.  It is convenient for our work.  We note, however, that more refined versions, using different integral quantities depending on $\gamma$, are stated in~\cite{SnelsonSolomon,HendersonSnelsonTarfulea}.

\subsubsection{The Boltzmann equation}

We now review the corresponding results for the Boltzmann equation. Here, $\Q = \QB$ is given by~\eqref{eq:Boltzmann}, and we consider the ``non-cutoff'' kernel, meaning that~\eqref{e.collision_kernel} holds for some $s\in (0,1)$.  In this case, the Boltzmann equation has smoothing properties~\cite{imbert2020regularity,Silvestre_Review}.  This allows one to bootstrap {\em a priori} bounds to a continuation criterion.  One such result, coming from~\cite{henderson2025decay,HST_Boltzmann_existence} is the following.

\begin{theorem}\label{t:local-well-posedness-Boltzmann-soft}
Assume that $\gamma \in (-3,0)$, $s\in (0,1)$, and $\Q = \QB$ as in~\eqref{eq:Boltzmann}.  Let $B$ be as in~\eqref{e.collision_kernel}. Assume that $f_{\rm in}$ is a continuous nonnegative function such that $f_{\rm in} \in L^\infty_m$ with $m > 3 + 2s$.

\begin{enumerate}

    \item There exists $T>0$ and a classical solution $f : [0,T]\times \Omega\times \R^3 \to [0,\infty)$. The function $f$ is strictly positive for $t>0$. The function $f$ is continuous up to $t = 0$ and $f(0,x,v) = f_{\rm in}(x,v)$.

    \item For any $k$, there is $m_k$ such that, if $m > m_k$, then $D_{t,x,v}^k f$ is continuous and bounded.

    \item There is $m_{\gamma,s} > 0$, depending only on $\gamma$ and $s$, such that the solution can be extended for as long as $\|f(t)\|_{L^\infty_{m_{\gamma,s}}} < +\infty$ with finite $L^\infty_m$-norm.

\end{enumerate}
\end{theorem}

\subsection{Notion of solution}\label{ss.notion_of_solution}

As illustrated in the previous subsection, very little regularity is needed to construct solutions and these solutions regularize immediately.  On the other hand, the uniqueness theory lags the existence theory.  Typically greater regularity and/or stronger non-vacuum conditions are assumed in order to obtain uniqueness~\cite{HendersonSnelsonTarfulea,HST_Boltzmann_existence,AMUXY1,AMUXY3,AMUXY4,HendersonSnelsonTarfulea1}.

Given the varied landscape, we opt to balance simplicity of our statements and proofs with the generality of our theorems.  As a result, we make stronger assumptions than is truly required for our proofs.  The alternative would involve significantly longer and more technical arguments, frequent discussions of slight differences between the Landau and Boltzmann cases, handling weak notions of solutions and their uniqueness in the low-regularity setting, and other added details that obscure the simple {\em a priori} estimate that is the heart of this work.

In this paper, we work with classical solutions in the sense that $(\partial_t + v\cdot \nabla_x) f$ and $D^2_vf$ are continuous.  We take $f$ to be continuous up to $t=0$.  We assume that $f$ has slightly improved qualitative decay; that is, for some $\eps>0$, $\vv^{m+\eps} f$ tends to zero as $|v|\to\infty$ uniformly in $(t,x)$.  In the whole space case $\Omega = \R^d$, we also assume that $\vv^m f$ decays as $|x|\to \infty$.

These assumptions can be removed with standard arguments.  For example, the extra qualitative decay in $L^\infty_{m+\eps}$ can be handled directly using, e.g., the ideas in \cite[Section 5]{ImbertMouhotSilvestre_decay}.  Note that $\eps$ does not enter the estimate in Theorem \ref{thm:main_propagation} quantitatively.  Another avenue is to use an approximation procedure as in~\cite{HendersonSnelsonTarfulea,HST_Boltzmann_existence}.  Indeed, the approximation procedure in those works shows that the constructed solutions starting from low regularity initial data are ``close'' to smooth solutions that remain in $L^\infty_{m'}$ for any large $m'$.  Thus, one obtains Theorem \ref{thm:main_propagation} for the smooth solutions and passes them to the low regularity setting in the limit.

\section{Propagation of upper bounds}
\label{sec:propagation}

We prove Theorem \ref{thm:main_propagation} by contradiction, using a comparison principle argument adapted to our setting. Consequently, the core of the argument is constructing appropriate barrier functions. In this section, we introduce our barrier function $\bb$ and isolate the main estimate on $\bb$ from which we deduce Theorem \ref{thm:main_propagation}.

We use a barrier function $\bb$ which is comparable to $\alpha(t) \brak{v}^{-m}$, where $\alpha$ grows in time according to $\|f(t)\|_{L^\infty_{d+\gamma}}$ and $m$ is a large exponent. Additionally, it is convenient for scaling for our computations that $\bb(v)$ is exactly equal to $\alpha |v|^{-m}$ for large $|v|$.
To that end, we define
\begin{equation}\label{eq:barrier}
    \bb(v) \coloneqq \alpha\bb_1(v) \qquad \text{where} \qquad \bb_1 (v) \coloneqq \begin{cases}
        |v|^{-m} &\text{ if } |v| \geq 1/2\\
        \text{smooth and positive} &\text{ if } |v| < 1/2.
    \end{cases}
\end{equation}
We further assume the function $\bb_1$ is radially symmetric, non-increasing in $|v|$, and $\bb_1(v) \leq |v|^{-m}$ for all $v$. Note that we have $\bb_1(v) \approx \brak{v}^{-m}$ for all $v$, but the function $\bb_1$ has simpler scaling properties for large $|v|$ that are convenient for our computations.

The main property of the barrier $\bb$ is contained in the following lemma used to prove Theorem \ref{thm:main_propagation}:
\begin{lemma} \label{lem:main_lemma}
    Let $\Q$ be either the Landau collision operator as in~\eqref{eq:Landau} or the Boltzmann collision operator as in~\eqref{eq:Boltzmann} with a kernel satisfying~\eqref{e.collision_kernel}. Then, there is an $m_0$, depending only on $\Q$, such that the following statement holds:

    For any $m > m_0$, define $\bb$ using this value of $m$ in \eqref{eq:barrier}, and suppose $f \in C^2(\R^d)$ satisfies
    \begin{equation}
        0\le f(v) \le \bb(v) \quad\text{for all }v\in \R^d \qquad\text{and} \qquad f(v_0) = \bb(v_0) \quad \text{for some }v_0\in\R^d.
    \end{equation}
    Then,
    \begin{equation*}
        \Q(f,f)(v_0) \le C \bb(v_0)^2 \brak{v_0}^{d + \gamma},
    \end{equation*}
    where $C$ depends only on $m$ and $\Q$.
\end{lemma}

As mentioned in the introduction, $m_0 = d+2+\gamma$ in the Landau case.  It is more complicated in the Boltzmann case, and, as such, we do not find it explicitly.

In the forthcoming comparison principle argument, $f$ will be a solution to Landau or Boltzmann equation, $\bb$ will be the barrier, and the point $v_0$ will be a contact point between the two. A simple corollary of Lemma \ref{lem:main_lemma} is the following estimate at any contact point $v_0$:
\begin{equation} \label{eq:main_lemma_consequence}
    \Q(f,f)(v_0) \leq C \bb(v_0) \|f\|_{L^\infty_{d+\gamma}}.
\end{equation}
The inequality \eqref{eq:main_lemma_consequence} follows from the simple observation that at any contact point:
\begin{equation}
\brak{v_0}^{d + \gamma} \bb(v_0) = \brak{v_0}^{d + \gamma} f(v_0) \leq \|f\|_{L^\infty_{d+\gamma}}.
\end{equation}
We now use \eqref{eq:main_lemma_consequence} to prove Theorem \ref{thm:main_propagation}.

\begin{proof}[Proof of Theorem \ref{thm:main_propagation}]
    Fix $m > m_0$ and $\eps>0$. Let $C$ be the associated constant from Lemma \ref{lem:main_lemma}.
    We show that, for any $t \in [0,T]$, any $x\in \Omega$, and any $v \in \R^d$,
    \begin{equation}\label{eq:eps_bound}
        f(t,x,v)
        < \underbrace{\left(\|f_{\rm in}\|_{L^\infty_m} + \eps\right) \exp\left[(C+\eps) \int_0^t \|f(s)\|_{L^\infty_{d+\gamma}} \dd s\right]}_{=:\alpha(t)} \bb_1(v)
        \eqqcolon \bb(t,v),
    \end{equation}
    which is sufficient to conclude the proof after passing $\eps \to 0^+$. 

    We argue by contradiction. Suppose that the inequality~\eqref{eq:eps_bound} fails.  By the qualitative regularity and decay properties of $f$ (see Section \ref{ss.notion_of_solution}), the set of contact points
    \begin{equation}
        \mathrm{CP} \coloneqq \set{(t,x,v) \in [0,T]\times\Omega \times \R^d  \; \mid \; f(t,x,v) = \bb(t,v)}
    \end{equation}
    is non-empty and compact.  
    Thus, there is an ``earliest'' contact point $(t_0,x_0,v_0) \in \mathrm{CP}$; that is, a contact point with the smallest time coordinate. From~\eqref{eq:eps_bound} and the positivity of $\eps$, we see that $t_0 > 0$.  Because this is the earliest contact point, we also see that
    \begin{equation*}
        f(t_0,x_0,v_0)
        = \bb(t_0,v_0) \qquad \text{and} \qquad f(t,x,v) \le \bb(t,v) \qquad \text{for each }0\le t \le t_0\;\text{and}\;(x,v)\in\Omega\times \R^d. 
    \end{equation*}
    Applying Lemma \ref{lem:main_lemma}, or rather \eqref{eq:main_lemma_consequence}, at the fixed time $t_0$, we find that
    \begin{equation}
        \Q(f,f)(t_0,x_0,v_0)  \leq C \bb(t_0,v_0) \|f(t_0)\|_{L^\infty_{d+\gamma}}.
    \end{equation}
    We now show that this contradicts~\eqref{eq:kinetic}. 
    Indeed, the minimum of $\bb - f$ over the set $[0,t_0]\times \Omega \times \R^d$ is obtained at $(t_0,x_0,v_0)$, implying that
    \begin{equation}
        (\partial_t + v\cdot\nabla_x)f(t_0,x_0,v_0)
        \geq (\partial_t + v\cdot\nabla_x) \bb(t_0,v_0)
        = (C+\eps) \|f(t_0)\|_{L^\infty_{d+\gamma}}\bb(t_0,v_0).
    \end{equation}
    Recalling~\eqref{eq:kinetic}, we find
    \begin{equation}
        \begin{split}
            (C+\eps) \|f(t_0)\|_{L^\infty_{d+\gamma}}\bb(t_0,v_0) \leq (\partial_t + v_0\cdot\nabla_x)f &= \Q(f,f)(t_0,x_0,v_0)\\
            &\leq C \|f(t_0)\|_{L^\infty_{d+\gamma}} \bb(t_0,v_0),
        \end{split}
    \end{equation}
    which is clearly a contradiction. This concludes the proof.
\end{proof}

We have shown how Lemma \ref{lem:main_lemma} implies Theorem \ref{thm:main_propagation}. The remainder of this paper is devoted to the proof of Lemma \ref{lem:main_lemma} both for the Landau and Boltzmann equations.

\section{Local and nonlocal contributions to the collision operator}\label{sec:local_nonlocal_collisions}

In this section, we reduce Lemma \ref{lem:main_lemma} into two precise, generic statements about collision operators, Lemma \ref{lem:removing_small_w} and Lemma \ref{lem:crude_bound_large_w} below; the second is further reduced by scaling to Lemma \ref{lem:crude_bound_large_w_scaled}. These lemmas yield Lemma \ref{lem:main_lemma} without using fine properties of $\QL$ and $\QB$.  Verifying them, however, requires handling the Landau and Boltzmann cases separately due to the differences in the forms of $\QL$ and $\QB$.

The main difficulty in establishing Lemma \ref{lem:main_lemma} lies in properly accounting for the nonlocal structure of the collision operator $\Q$. The assumption $f(v) \leq \bb(v)$ says that the bulk of $f$ is concentrated around the origin in velocity space. It is well documented that nonlocal contributions of the bulk may unavoidably destroy decay properties. For example, consider the following decay estimate for the three-dimensional inverse Laplacian:
\begin{equation*}
    \sup_{v \in \R^3} \abs{\Delta^{-1}f(v)} = \sup_{v\in\R^3} \frac{1}{4\pi}\abs{\int_{\R^3} \frac{f(w)}{\abs{v-w}} \dd w} \lesssim \frac{\norm{f}_{L^\infty_m}}{\brak{v}} \qquad \text{for}\quad m > 3.
\end{equation*}
Relying only on decay properties of $f$, this estimate is optimal: choosing $m$ larger cannot improve the decay rate of $\Delta^{-1}f$.
Similar convolutions appear in the Landau equation, and even more severe nonlocal contributions occur in the Boltzmann equation. 
Clearly one must use precise structure of the collision operator to overcome this type of nonlocal phenomenon.

As we will see, the large values of $f(w)$ have a negative contribution to $\Q(f,f)(v)$ when $|w|$ is small. Once we understand that, we can use a crude estimate for the contribution of the values of $f(w)$ for $w$ large to obtain the scaling consistent estimate of Lemma \ref{lem:main_lemma}. This logic applies to both the Landau and Boltzmann equations, but the details of the computations are different in each case.

\subsection{The nonlocal contribution of small velocities}

The following lemma tells us that we can neglect the values of $f(w)$ for $w$ small, with the same assumptions as in Lemma \ref{lem:main_lemma}. We consider the same type of barrier function $\bb$ as in Section \ref{sec:propagation}.

\begin{lemma} \label{lem:removing_small_w}
    Let $\Q$ be either the Landau collision operator as in~\eqref{eq:Landau} or the Boltzmann collision operator as in~\eqref{eq:Boltzmann} with a kernel satisfying~\eqref{e.collision_kernel}. Then, there is an $m_0 \gg 1$ depending on $\Q$ such that the following statement holds:

    For any $m > m_0$ and any $\alpha > 0$, define $\bb$ as in \eqref{eq:barrier} and suppose $f \in C^2(\R^d)$ satisfies
    \begin{equation}
        0\le f(v) \le \bb(v) \quad\text{for all }v\in \R^d \qquad\text{and} \qquad f(v_0) = \bb(v_0) \quad \text{for some }v_0\in\R^d\text{ with } \abs{v_0} \ge 1.
    \end{equation}
    For any $\delta > 0$, consider the auxiliary function
    \begin{equation}\label{defn:fbar}
    \uf(v) = \begin{cases}
        f(v) &\text{ if } \abs{v} \geq \delta \abs{v_0}\\
        0 &\text{ if } \abs{v} < \delta\abs{v_0}
    \end{cases}.
    \end{equation}
    Then, for some $\delta \ll 1$ depending only on $m$ and $\Q$, we have
    \begin{equation*}
        \Q(f,f)(v_0) \le \Q(\uf,\uf)(v_0).
    \end{equation*}
\end{lemma}

\subsection{The contribution of large velocities}

In order to prove Lemma \ref{lem:main_lemma}, we will combine Lemma \ref{lem:removing_small_w} with the following lemma, which gives us a crude estimate for the contribution of the values of $f(w)$ for $w$ large.

\begin{lemma} \label{lem:crude_bound_large_w}
    Let $\Q$ be either the Landau collision operator as in~\eqref{eq:Landau} or the Boltzmann collision operator as in~\eqref{eq:Boltzmann} with a kernel satisfying~\eqref{e.collision_kernel}. Then, there is an $m_0 \gg 1$ depending on $\Q$ such that the following statement holds.

    For any $m > m_0$ and $\alpha > 0$, define $\bb$ as in \eqref{eq:barrier}. For any $\delta \in(0,1)$, $v_0\in\R^d$, and $f\in C^2(\R^d)$, define $\uf$ as in \eqref{defn:fbar}.
    Suppose that
    \begin{equation}
        0\le f(v) \le \bb(v) \quad\text{for all }v\in \R^d \qquad\text{and} \qquad f(v_0) = \bb(v_0) \quad \text{for some }v_0\in\R^d\text{ with } \abs{v_0} \ge 1.
    \end{equation}
    Then, we have that
    \begin{equation*}
        \Q(\uf,\uf)(v_0) \lesssim \alpha^2\abs{v_0}^{-2m + d + \gamma},
    \end{equation*}
    where the implied constant depends on $m$, $d$, $\delta$, and $\Q$.
\end{lemma}

\noindent Recall that both the Boltzmann and Landau collision operators satisfy the following scaling property:
\begin{equation}\label{eq:scaling}
    \Q(\alpha f_\lambda, \alpha f_\lambda)(v)
    = \alpha^2 \lambda^{-d-\gamma} \Q(f,f)(\lambda v) \qquad\text{ where }f_\lambda(w) = f(\lambda w).
\end{equation}
Consequently, the homogeneity of the barrier function $\bb$ for large velocities combined with scaling reduces Lemma \ref{lem:crude_bound_large_w} to the following simplified statement, which we apply to $\tilde f(v) \coloneqq \alpha^{-1} |v_0|^m \underline f(|v_0|v)$ and $e = \hat v_0$.

\begin{lemma} \label{lem:crude_bound_large_w_scaled}
    Let $\Q$ be either the Landau collision operator as in~\eqref{eq:Landau} or the Boltzmann collision operator as in~\eqref{eq:Boltzmann} with a kernel satisfying~\eqref{e.collision_kernel}. Then, there is an $m_0 > 0$ depending on $\Q$ such that the following statement holds.

    For any $m > m_0$, $\delta \in (0,1/2)$, $e \in \S^{d-1}$, and $f\in C^2(\R^d)$ satisfying
    \begin{equation*}
        0 \le f(v) \le \abs{v}^{-m} \; \text{for all }v\in \R^d, \qquad f(v) = 0 \;\; \text{for all }v\in B_\delta, \qquad \text{and} \qquad f(e) = 1,
    \end{equation*}
    then $$\Q(f,f)(e) \leq C_0,$$ for some constant $C_0$ depending on $m$, $d$, $\delta$, and $\Q$.
\end{lemma}

\begin{remark}
    For Lemma \ref{lem:crude_bound_large_w_scaled} and consequently Lemma \ref{lem:crude_bound_large_w}, the choice $m_0 =  d + \gamma + 2$ suffices. In particular, at this stage, $m_0$ is simple and explicitly computable for both Boltzmann and Landau. In the Boltzmann case, the complicated restriction on $m_0$  comes from Lemma \ref{lem:removing_small_w}.
\end{remark}

\subsection{The reduction of Lemma \ref{lem:main_lemma}}
Before proceeding, let us briefly show how to combine Lemma \ref{lem:removing_small_w} and Lemma \ref{lem:crude_bound_large_w_scaled} to obtain the main lemma. 

\begin{proof}[Proof of Lemma \ref{lem:main_lemma}]
We first consider the case where $|v_0| \geq 1$.  Combining Lemma \ref{lem:removing_small_w} and Lemma \ref{lem:crude_bound_large_w} yields
\begin{equation}
    \Q(f,f)(v_0)
    \leq \Q(\uf,\uf)(v_0)
    \lesssim \alpha^2\abs{v_0}^{-2m+d+\gamma}
    \approx \bb(v_0)^2 \vvo^{d+\gamma}.
\end{equation}
The proof is complete in this case.

The case when $\abs{v_0} \leq 1$ is simpler because we do not need to account for the precise growth of $\Q(f,f)$. Indeed, one needs only show that $\Q(f,f) \lesssim 1$. This follows by a coarse estimate as in the proof of Lemma \ref{lem:crude_bound_large_w_scaled}.
\end{proof}

\section{A priori estimates for the Landau equation}\label{sec:Landau_a_priori}

In this section, we establish Lemma \ref{lem:removing_small_w} and Lemma \ref{lem:crude_bound_large_w_scaled} for the Landau equation. The argument for the Landau equation contains fewer technical details due to the more local nature of the collision operator.

\subsection{The nonlocal contribution of small velocities: Landau case}

We begin with the proof of the more subtle of the two main lemmas, which states that we may neglect collisions with small velocities.

\begin{proof}[Proof of Lemma \ref{lem:removing_small_w} when $\Q=\QL$]
    Recall that the Landau collision operator takes the form
    \begin{equation}
        \QL(f,f)(v) = \bar a_{ij} \partial_{ij} f(v) + \bar c f(v),
    \end{equation}
    where $\bar a$ and $\bar c$ are the convolutions defined in~\eqref{eq:defn_coeff}.
    Since $f$ and $\uf$ coincide in a neighborhood of $v_0$, $\QL(f,f)(v_0)$ and $\QL(\uf,\uf)(v_0)$ differ only in the value of $\bar a_{ij}(v_0)$ and $\bar c(v_0)$:
    \begin{equation}
    \begin{aligned}
        \QL(f,f)(v_0) - \QL(\uf, \uf)(v_0) &=  \QL(f - \uf,f)(v_0) - \QL(\uf,f- \uf)(v_0)
        \\&
        = \bar a_{ij}[f - \uf] \partial_{ij} f(v_0) + \bar c[f - \uf] f(v_0)
        \\&
        = \bar a_{ij}[g] \partial_{ij} f(v_0) + \bar c[g] f(v_0),
    \end{aligned}
    \end{equation}
    where we have implicitly introduced the notation
    \begin{equation}
        g = f- \uf.
    \end{equation}
    Note that $g$ is a nonnegative function that is supported in the ball of radius $\delta\abs{v_0}$ centered at the origin, so that the coefficients take the form
    \begin{equation*}
    \begin{aligned}
        \bar a_{ij}[g](v_0) &= \int_{B_{\delta\abs{v_0}}} \abs{v_0-w}^{2+\gamma} \Pi_{ij}(v_0-w)g(w) \dd w\\[5pt]
        \bar c[g](v_0) &= \begin{cases}
        (d-1)(d+\gamma)\int_{B_{\delta\abs{v_0}}} \abs{v_0-w}^\gamma g(w) \dd w &\text{if } \gamma > -d\\[4pt]
        0&\text{if } \gamma = -d,
        \end{cases}
    \end{aligned}
    \end{equation*}
    where the simplification when $\gamma = -d$ occurs because $\bar c[g]$ is a local operator in that case.
    By assumption, $\bb(v) - f(v)$ is a $C^2$ function that attains its minimum value at $v_0$.
    Thus,
    \begin{equation}
        f(v) \le \bb(v) \quad \text{for all }v\in\R^d \qquad \text{and} \qquad \nabla^2 f(v_0) \le \nabla^2 \bb(v_0) = \frac{m \bb(v_0)}{\abs{v_0}^2}\left[\frac{(m+2)v_0\otimes v_0}{\abs{v_0}^2} - \Id\right],
    \end{equation}
    using the specific form of $\bb(v_0)$ in \eqref{eq:barrier} and that $\abs{v_0} \ge 1$.
    Therefore,
    \begin{equation}\label{eq:integral_Landau}
    \begin{aligned}
        \QL(f,f)(v_0) &- \QL(\uf, \uf)(v_0) \leq \bar a_{ij}[g] \partial_{ij} \bb(v_0) + \bar c[g] \bb(v_0) \\
        &= \int_{B_{\delta \abs{v_0}}} f(w) \left( \abs{v_0-w}^{\gamma+2} \Pi_{ij}(v_0-w) \partial_{ij} \bb(v_0) + (d-1)(d+\gamma)|v_0-w|^\gamma \bb(v_0) \right) dw
    \end{aligned}
    \end{equation}
    To finish the proof, we claim that for any $m$ sufficiently large, there is a $\delta$ such that
    \begin{equation}\label{eq:integrand_Landau}
        \abs{v_0-w}^{\gamma+2} \Pi_{ij}(v_0-w) \partial_{ij} \bb(v_0) + (d-1)(d+\gamma)|v_0-w|^\gamma \bb(v_0) \leq 0, \qquad \text{for each }w \in B_{\delta\abs{v_0}}.
    \end{equation}
    As soon as we verify this claim, the integrand in \eqref{eq:integral_Landau} is non-positive and $\QL(f,f)(v_0) - \QL(\uf, \uf)(v_0) \leq 0$, which is the desired conclusion.

    We are left to verify the claim~\eqref{eq:integrand_Landau}. First, notice that $\Pi$ is a rank $d-1$ orthogonal projection, which implies that $\mathrm{tr}(\Pi) = d-1$.  Combined with the form of $\nabla^2\bb(v_0)$, we see that~\eqref{eq:integrand_Landau} reduces to the claim that
    \begin{equation}
    \label{eq:integrand_Landau_reduced}
        \frac{m\abs{v_0 - w}^{2}}{\abs{v_0}^2} \left[\frac{(m+2)\Pi(v_0-w)v_0 \cdot v_0}{\abs{v_0}^2} - (d-1)\right] + (d-1)(d+\gamma) \le 0, \qquad \text{for each }w\in B_{\delta\abs{v_0}}
    \end{equation}
    Since \eqref{eq:integrand_Landau_reduced} is invariant under rotations and scaling of $v_0$, we may assume $v_0 = e$, a fixed unit vector, which reduces \eqref{eq:integrand_Landau} further to the claim that
    \begin{equation}\label{eq:integrand_Landau_final}
        m\abs{e - w}^{2} \left[(m+2)\Pi(e-w)e \cdot e - (d-1)\right] + (d-1)(d+\gamma) \le 0, \qquad \text{for each }w\in B_{\delta}.
    \end{equation}
    Since $\Pi(e)e = 0$, \eqref{eq:integrand_Landau_final} is verified for $w = 0$, provided $m > d + \gamma$. Additionally, the left hand side of \eqref{eq:integrand_Landau_final} is evidently continuous in $w$. As a consequence, for $m>d+\gamma$, there exists $\delta > 0$ so that \eqref{eq:integral_Landau} holds.  This concludes the proof of Lemma \ref{lem:removing_small_w} for the Landau equation.
\end{proof}

\subsection{The contribution of large velocities: Landau case}
We proceed to prove Lemma \ref{lem:crude_bound_large_w_scaled} for the Landau equation. There is nothing subtle about this proof. We merely find a coarse estimate for $\QL(\uf, \uf)(e)$, where $e$ is a unit vector.

\begin{proof}[Proof of Lemma \ref{lem:crude_bound_large_w_scaled} when $\Q=\QL$]
    We will bound each term in the nondivergence form of $\QL$:
    \begin{equation*}
        \QL(f, f)(e) = \bar a_{ij}[f] \partial_{ij} f(e) + \bar c[f] (e).
    \end{equation*}
    Note first that $f(e) = 1$. Thus, we only need to estimate $\bar a_{ij}[f]$, $\bar c[f]$, and $\nabla^2 f(e)$. We begin with $\nabla^2f(e)$. By assumption, $f-\bb(v)$ is a $C^2$ function that obtains its maximum value at $e$, which implies
    \begin{align*}
        \nabla^2 f(e)
        &\leq \nabla^2 \bb(e) = m(m+2)(e\otimes e) - m\Id,
    \end{align*}
    where the inequality holds as symmetric matrices.
    We proceed to estimate $\bar a_{ij}[f]\partial_{ij}f$:
    \begin{align*}
        \bar a_{ij}[f]
        \partial_{ij} f(e)
        &\le m\int_{\R^d} \abs{e-w}^{2+\gamma} \left((m+2)\Pi(e-w)e\cdot e - (d-1)\right) f(w) \dd w
        \\&
        \leq m(m+2)\int_{\R^d\setminus B_\delta} \abs{e-w}^{2+\gamma} \abs{w}^{-m} \dd w
        \lesssim 1,
    \end{align*}
    provided that $m > d + \gamma +2$.
    Lastly, we estimate $\bar c[f]$. In the case $\gamma = -d$, we simply have $\bar c[f](e) \approx f(e) = 1$. For $\gamma > -d$, we have
    \begin{align*}
        \bar c[f] &= (d-1) (d+\gamma) \int_{\R^d} \abs{e-w}^\gamma f(w) \dd w \\
        &\leq (d-1) (d+\gamma) \int_{\R^d \setminus B_\delta} \abs{e-w}^\gamma |w|^{-m} \dd w
        \lesssim 1,
    \end{align*}
    provided that $m > d + \gamma$. Combining these bounds concludes the proof for the Landau equation.
\end{proof}

\section{A priori estimates for the Boltzmann equation}\label{sec:Boltzmann_a_priori}

We are left to prove Lemma \ref{lem:removing_small_w} and Lemma \ref{lem:crude_bound_large_w_scaled} for the Boltzmann equation. This case is more involved than that of the Landau equation because of the more nonlocal structure of the Boltzmann collision operator. In particular, the delicate step is properly isolating \emph{all} of the nonlocal contributions in the proof of Lemma \ref{lem:removing_small_w}.

For our initial computations, let us assume the angular cross section $b$ is integrable, allowing us to separate the collision operator into gain and loss terms and manipulate each individually. We will derive an equality for the difference $\QB(f,f)(v_0) - \QB(\uf,\uf)(v_0)$, which will be the starting point for our analysis. This equality holds for any collision cross section $b$, including non-integrable ones. The latter can be seen by a limiting procedure, approximating non-integrable cross sections by a sequence of integrable ones.

\subsection{The nonlocal contribution of small velocities: Boltzmann case}

\begin{proof}[Proof of Lemma \ref{lem:removing_small_w} when $\Q=\QB$]
We proceed in a series of steps.  The main idea is similar to that of the Landau case.  We rewrite the difference of the collision operators until it is obvious that the claim is true when $\delta=0$ and conclude by continuity.

\medskip
\noindent
{\bf Step 1: a change of variables.}
We begin by expanding the difference $\QB(f,f)(v_0) - \QB(\uf,\uf)(v_0)$. Let us recall that
\begin{equation}
    \underline{f}(v_0) = f(v_0)
    \qquad \text{ and }\qquad
    \underline{f}(w) = f(w)
        \quad\text{ if } |w|\geq \delta |v_0|.
\end{equation}
Hence, we obtain
\begin{equation*}
\begin{split}
    \QB(f,f)&(v_0) - \QB(\uf,\uf)(v_0)\\
    &\quad= \iint [f(v')f(v'_*) - f(v_0)f(v_*)]B \dd \sigma \dd v_* - \iint \left[\underline{f}(v')\underline{f}(v_*') - f(v_0)\underline{f}(v_*) \right]B \dd \sigma \dd v_*\\
    &\quad = - f(v_0)\iint \left[f(v_*) - \underline{f}(v_*)\right] B \dd \sigma \dd v_* + \iint \left[f(v') - \underline{f}(v')\right]f(v_*') B \dd \sigma \dd v_*\\
    &\quad\qquad+\iint \underline{f}(v')\left[f(v_*') - \underline{f}(v_*')\right]B\dd \sigma \dd v_*\\
    &\quad = - f(v_0)\iint_{v_* \in B_{\delta\abs{v_0}}}\limits f(v_*) B \dd \sigma \dd v_* + \iint_{v' \in B_{\delta\abs{v_0}}}\limits f(v')f(v_*') B \dd \sigma \dd v_*
        +\iint_{\substack{{v_*' \in B_{\delta\abs{v_0}}}\\v'\notin B_{\delta\abs{v_0}}}}\limits f(v')f(v_*')B\dd \sigma \dd v_*\\[-10pt]
    &\quad\coloneqq I_1 + I_2 + I_3.
\end{split}
\end{equation*}
In the last integral, the second condition $v' \notin B_{\delta\abs{v_0}}$ is redundant since $v'$ and $v_*'$ cannot both belong to $B_{\delta\abs{v_0}}$ if $\delta < 1/2$; see~\eqref{e.v'_v_*'_sigma}. Observe that $I_1<0$.  The goal of the proof, then, is to show that $I_1$ dominates. This is plausible because $I_1$ contains $f(v_0)$, which is ``large'' in an appropriate sense compared to $f(v)$ for other $v\neq v_0$.  More precisely, $v_0$ is the location of a minimum of the function $\bb - f$.

Next, we change variables in each integral to factor out the values $f(w)$ for $w \in B_{\delta \abs{v_0}}$.  The change of variables used in each integral $I_i$ is different. For $I_1$, we use a stereographic projection. For $I_2$ and $I_3$, we will use the Carleman coordinates introduced in Section \ref{sec:preliminaries}; these changes of variables are classical and references for the associated Jacobians are provided there.

We start with $I_2$, where we change variables to integrate with respect to $v_*'$ and $v'$:
\begin{align*}
    I_2 &= \int_{B_{\delta \abs{v_0}}} f(v')\int_{(v' - v_0) \perp (v'_* - v_0)}  f(v'_*) \abs{v-v_*}^\gamma \ b(\sin\theta/2) \frac{2^{d-1}}{r^{d-2} |v_0-v'|} \dd v'_*\dd v' \\
    &= 2^{d-1} \int_{B_{\delta \abs{v_0}}} f(v')\int_{(v' - v_0) \perp (v'_* - v_0)}  f(v'_*) \ \frac{r^{2-d+\gamma} }{|v_0-v'|} \ b\left(\sin(\theta/2)\right) \dd v'_*\dd v',
\end{align*}
where $r$ is the diameter of the associated collision sphere:
\begin{equation}
    r^2 = \abs{v_0-v'}^2 + \abs{v_0 - v'_*}^2 = |v'-v_*'|^2 = \abs{v-v_*}^2 \qquad \text{and} \qquad \sin(\theta/2) = \frac{|v_0-v'|}{r}
\end{equation}
For consistency with later integrals, we rename the variables of integration as $(w,z) \coloneqq (v',v'_*)$:
\begin{equation}
        I_2 = 2^{d-1}\int_{B_{\delta \abs{v_0}}} f(w)\int_{(w - v_0) \perp (z - v_0)}  f(z)  \ \frac{r^{2-d+\gamma}}{|v_0-w|} \ b\left(\frac{|v_0-w|}{r}\right) \dd z\dd w.
\end{equation}
In these new variables, we note that $r=r(z,w)$ takes the form
\begin{equation}\label{e.r_theta_vw}
    r^2
    = \abs{v_0-w}^2 + \abs{v_0 - z}^2
    = \abs{z-w}^2.
\end{equation}
For $I_3$, we perform a similar change of variables but for the outside integration with respect to $v'_*$:
\begin{equation}
\begin{aligned}
    I_3 &= 2^{d-1} \int_{B_{\delta \abs{v_0}}} f(v'_*)\int_{(v' - v_0) \perp (v'_* - v_0)}  f(v') \  \frac{r^{2-d+\gamma} }{|v_0-v'_\ast|} \ b\left(\sin\theta/2\right) \dd v'\dd v'_*.
\end{aligned}
\end{equation}
Renaming the variables of integration as $(w,z) \coloneqq (v'_*,v')$ and using the same formula for $\sin(\theta/2)$ and $r$ results in
\begin{equation*}
\begin{aligned}
    I_3
    &= 2^{d-1}\int_{B_{\delta \abs{v_0}}} f(w)\int_{(z - v_0) \perp (w - v_0)}  f(z) \frac{r^{2-d+\gamma}}{|v_0-w|} \  b\left(\frac{|v_0-z|}{r}\right) \dd z\dd w.
\end{aligned}
\end{equation*}
It remains to change variables in $I_1$. For $I_1$, instead of Carleman coordinates, we use a stereographic projection to the same hyperplane that $I_2$ and $I_3$ are integrated over. First, we rename $w = v_*$ so that $I_1$ becomes
\begin{equation*}
    I_1 = - f(v_0)\int_{B_{\delta\abs{v_0}}} f(w) \int_{\S^{d-1}} \abs{v_0-w}^\gamma \ b(\sin\theta/2) \dd \sigma \dd w.
\end{equation*}
Second, the vector $\sigma$ is mapped by stereographic projection from $\S^{d-1}$ taken with north pole $w$ and south pole $v_0$ to a point $z$ on the plane through $v_0$ orthogonal to $v_0 - w$, i.e. $(z - v_0) \perp (w -v_0)$.  See Figure \ref{fig:stereographic_projection}.
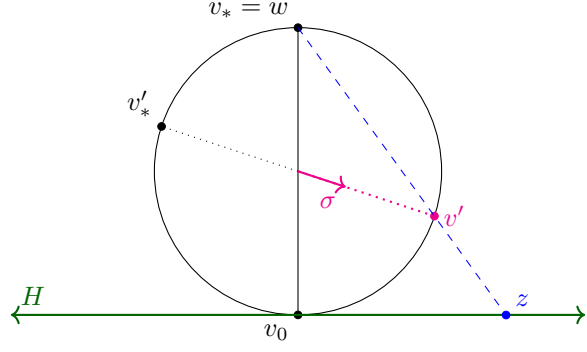
\begin{figure}
\centering
\begin{tikzpicture}[scale=.95]
    \def\r{2} 
    \coordinate (W) at (0,\r);
    \coordinate (V) at (0,-\r);
    \coordinate (V') at (.95*\r,-.3122*\r); 
    \coordinate (VSTAR') at ($-1*(V')$);
    \coordinate (H) at (1.5,\r+0.5);
    \draw (0,0) circle (\r);

    \draw[fill] (W) circle (1.5pt) node [above left] {$v_* = w$};
    \draw[fill] (V) circle (1.5pt) node [below left] {$v_0$};
    \draw[-] (0, \r) -- (0, -\r);

    \draw[fill, magenta] (V') circle (1.5pt) node[right] {$v'$};
    \draw[dotted, thick, magenta] (0,0) -- (V');
    \draw[fill] (VSTAR') circle (1.5pt) node[above left] {$v_*'$};
    \draw[dotted] (0,0) -- (VSTAR');
    \draw[->, magenta, thick] (0,0) -- ($.35*(V')$) node[below left] {$\sigma$};

    \coordinate (P1) at (-4, -\r);
    \coordinate (P2) at (4, -\r);
    \draw[thick,<->,DarkGreen] (P2) -- (P1) node[above right] {$H$};

    \coordinate (ZDIR) at ($(V')-(W)$);
    \draw[blue, dashed] (W) -- ($(W)+1.525*(ZDIR)$);
    \draw[fill, blue] ($(W) + 1.525*(ZDIR)$) circle (1.5pt) node[above right] {$z$};

    \draw[thick,-,opacity=0] (P1) -- (P2) node[midway, below] {$H_{\rm near}$};

\end{tikzpicture}
\caption{A figure showing the stereographic projection used in defining the integral $I_1$.  We note that the sphere in this figure is the Boltzmann sphere containing the original variables $v_0$, $v_* = w$, $v'$, and $v_*'$ whereas $H$ is the hyperplane $v_0 + w^\perp$.}
\label{fig:stereographic_projection}
\end{figure}
Explicitly, we have the formula relating $v_0$, $w$, $\sigma$, and the new coordinate $z$:
\begin{equation*}
    z = w + \frac{v_0 - w + \abs{v_0 - w}\sigma}{1 + \sigma \cdot \left(\dfrac{v_0-w}{\abs{v_0 - w}}\right)},
\end{equation*}
resulting in the following expression for $I_1$:
\begin{equation*}
\begin{aligned}
    I_1
    &= - f(v_0) \int_{B_{\delta\abs{v_0}}} f(w) \int_{(z-v_0)\perp (w-v_0)} \abs{w-v_0}^\gamma \ b(\sin \theta/2) \frac{2^{d-1}\abs{v_0-w}^{d-1}}{\abs{z-w}^{2(d-1)}} \dd z \dd w.
\end{aligned}
\end{equation*}
Here, we observe that $\sin(\theta/2) = |v_0-z|/|z-w|$. We write $r = |z-w|$ for consistency with the other integrals. We also note that $\sin(\theta/2) = |v_0-z| / r$. The final expression for $I_1$ is
\begin{equation*}
\begin{aligned}
    I_1
    &= - 2^{d-1} \int_{B_{\delta\abs{v_0}}} f(w) \int_{(z-v_0)\perp (w-v_0)} f(v_0) \ b\left(\frac{\abs{v_0-z}}{r}\right) \frac{\abs{v_0-w}^{\gamma + d-1}}{r^{2(d-1)}} \dd z \dd w.
\end{aligned}
\end{equation*}
Taken together, the changes of variables to the three integrals yield the following expression:
\begin{equation} \label{eq:I123}
\begin{split}
    &\QB(f,f)(v_0) - \QB(\uf,\uf)(v_0) = I_1 + I_2 + I_3 \\&
    \qquad
    = 2^{d-1}\int_{B_{\delta\abs{v_0}}} \frac{f(w)}{|v_0-w|} \int_{(z-v_0)\perp(w-v_0)}  \Bigg[\left(f(z) r^{2-d+\gamma} - f(v_0)\frac{\abs{v_0-w}^{\gamma + d}}{r^{2(d-1)}} \right) \ b\left(\frac{\abs{v_0-z}}{r}\right)
    \\& \hspace{3.5in}
    + f(z) r^{2-d+\gamma} \ b\left(\frac{\abs{v_0-w}}{r}\right)\Bigg] \dd z \dd w .
\end{split}
\end{equation}
We point out that while $I_1$ and $I_3$ would both be $+\infty$ in the case of a non-cutoff kernel, the final expression on the right hand side above is well-defined and finite: The singularity of the kernel is exactly canceled by the vanishing of the difference
\begin{equation}\label{e.c041801}
    f(z)r^{2-d+\gamma} - f(v_0)\frac{\abs{v_0-w}^{\gamma + d}}{r^{2(d-1)}}
\end{equation}
as $z\to v_0$ along the hyperplane $\set{z \mid (z-v_0)\cdot (w-v_0) = 0}$.

\medskip
\noindent
{\bf Step 2: simplifying the main inequality by scaling.}
In order to prove that $\QB(f,f)(v_0) - \QB(\uf,\uf)(v_0)$ is negative for $\delta$ sufficiently small, we analyze the expression \eqref{eq:I123}. The values of $f(w)$ for $w \in B_{\delta |v|}$ are arbitrary positive numbers. Hence, we need to show that the inner integrals are negative for $\delta$ sufficiently small. Indeed, we see that the proof of Lemma \ref{lem:removing_small_w} is complete once we establish the negativity of
\begin{equation}\label{eq:goal_ineq}
\begin{aligned}
    &\int_{(z-v_0)\perp(w-v_0)}  \Big[\Big(f(z) r^{2-d+\gamma} - f(v_0)\frac{\abs{v_0-w}^{\gamma + d}}{r^{2(d-1)}} \Big) \ b\Big(\frac{\abs{v_0-z}}{r}\Big) + f(z) r^{2-d+\gamma} \ b\Big(\frac{\abs{v_0-w}}{r}\Big)\Big] \dd z
\end{aligned}
\end{equation}
for $m$ sufficiently large and $\delta$ sufficiently small.  We recall the $m$ dependence in~\eqref{eq:goal_ineq} is through the definition of $\bb$ and the assumption that $f \leq \bb$.

Here, the key assumptions are that $f(v_0) = \bb(v_0)$ and $f(z) \leq \bb(z)$ for every $z \in \R^d$.
Moreover, since we assume that $|v_0|\geq 1$ and $|w| < \delta |v_0|$ with $\delta$ small, each hyperplane uniformly avoids $0$. More precisely, for $\delta$ sufficiently small, $|z| \geq 1/2$ for all $(z-v_0) \cdot (w-v_0) = 0$. Recalling the form of the barrier from ~\eqref{eq:barrier},
\begin{equation}
    f(z) \leq \bb(z) = \alpha |z|^{-m}
        \qquad\text{ for all } (z-v_0)\perp (w-v_0),
\end{equation}
and
\begin{equation}
    f(v_0) = \bb(v_0) = \alpha |v_0|^{-m}.
\end{equation}
Plugging these into the expression~\eqref{eq:goal_ineq} yields
\begin{align*}
    &\int_{(z-v_0)\perp(w-v_0)}  \Big[\Big(f(z) r^{2-d+\gamma} - f(v_0)\frac{\abs{v_0-w}^{\gamma + d}}{r^{2(d-1)}} \Big) \ b\Big(\frac{\abs{v_0-z}}{r}\Big) + f(z) r^{2-d+\gamma} \ b\Big(\frac{\abs{v_0-w}}{r}\Big)\Big]  \dd z \\
    &\leq
    \alpha \int_{(z-v_0)\perp(w-v_0)}  \Big[\Big( |z|^{-m} r^{2-d+\gamma} -  |v_0|^{-m}\frac{\abs{v_0-w}^{\gamma + d}}{r^{2(d-1)}} \Big) \ b\Big(\frac{\abs{v_0-z}}{r}\Big) +  |z|^{-m} r^{2-d+\gamma} \ b\Big(\frac{\abs{v_0-w}}{r}\Big)\Big] \dd z.
\end{align*}
Recalling that $r = \abs{z-w}$, the right hand side scales simply provided $z$, $w$, and $v_0$ are scaled identically. Indeed, making the change of variables $\widetilde v_0 \coloneqq \lambda v_0$, $\widetilde w \coloneqq \lambda w$, and $\widetilde z \coloneqq \lambda z$, the right hand side is simply multiplied by $\lambda^{\gamma - m + 1}$. Hence, we can assume without loss of generality that $v_0$ is a unit vector, i.e., $v_0 = e$, and establish non-positivity in that case. Likewise, we assume without loss of generality that $\alpha = 1$. Hence, we need to show that
\begin{equation}\label{eq:goal_ineq_scaled}
\begin{aligned}
    &\int_{(z-e)\perp(w-e)}  \left[\left(|z|^{-m} r^{2-d+\gamma} - \frac{\abs{e-w}^{\gamma + d}}{r^{2(d-1)}} \right) \ b\left(\frac{\abs{e-z}}{r}\right) +  |z|^{-m} r^{2-d+\gamma} \ b\left(\frac{\abs{e-w}}{r}\right)\right] \dd z \leq 0.
\end{aligned}
\end{equation}
The proof of Lemma \ref{lem:removing_small_w} is reduced to proving \eqref{eq:goal_ineq_scaled} for $|w|<\delta$ for $\delta$ sufficiently small.

\medskip
\noindent
{\bf Step 3: obtaining the inequality by continuity.}
The expression on the left-hand side of \eqref{eq:goal_ineq_scaled} is continuous with respect to $w$ around the origin. There is no difficulty in verifying this by analyzing the integrand term-by-term and using the dominated convergence theorem, even though writing down the full details is somewhat tedious. In the case of a non-cutoff cross section $b$, we have to pay attention to the cancellation of the singularity of the kernel with the vanishing of the difference
\begin{equation}
    |z|^{-m} \abs{z-w}^{2-d+\gamma} - \frac{\abs{e-w}^{\gamma + d}}{\abs{z - w}^{2(d-1)}}
\end{equation}
as $z \to e$. Due to this continuity, we can verify \eqref{eq:goal_ineq_scaled} by checking that for $m$ sufficiently large, the left-hand side is negative at $w=0$. Note that when $w=0$, $|z| = r$. Hence, we need to show that
\begin{equation}\label{eq:goal_ineq_origin}
\begin{aligned}
    &\int_{(z-e)\perp e}  \left[\left(|z|^{-m + 2-d+\gamma} - \abs{z}^{-2(d-1)} \right) \ b\left(\frac{\abs{e-z}}{|z|}\right) +  |z|^{-m + 2-d+\gamma} \ b\left(\frac{1}{|z|}\right)\right] \dd z < 0.
\end{aligned}
\end{equation}

Recall that the domain of integration is the hyperplane $H_0 = e + e^\perp$. On this hyperplane, $|z|>1$ except when $z=e$. The term $|z|^{-m + 2-d+\gamma}$ converges to zero monotonically as $m \to \infty$ almost everywhere on $H_0$. The negative term $-|z|^{-2(d-1)}$ is independent of $m$. Hence, by the monotone convergence theorem, the left-hand side of \eqref{eq:goal_ineq_origin} converges to
\begin{equation*}
\begin{aligned}
    &\int_{(z-e)\perp e}  -|z|^{-2(d-1)} \ b\left(\frac{\abs{e-z}}{|z|}\right) \dd z < 0.
\end{aligned}
\end{equation*}
This verifies \eqref{eq:goal_ineq_origin} (and, thus, also \eqref{eq:goal_ineq_scaled} and \eqref{eq:goal_ineq}) and concludes the proof of Lemma \ref{lem:removing_small_w}.
\end{proof}

\subsection{The contribution of large velocities: Boltzmann case.}

We now prove Lemma \ref{lem:crude_bound_large_w_scaled} for the Boltzmann equation. We again find a crude estimate for $\QB(f,f)(e)$, where $e$ is a unit vector and $f$ is any function satisfying the assumptions of Lemma \ref{lem:crude_bound_large_w_scaled}. This time, we will use the standard decomposition of Boltzmann operator $\Q = \Qs + \Qns$ introduced in Section \ref{sec:preliminaries}.

\begin{proof}[Proof of Lemma \ref{lem:crude_bound_large_w_scaled} when $\Q=\QB$]
We will bound both terms in the decomposition of $\QB$ separately:
\begin{equation}
    \QB(f,f)(e) = \Qs(f,f)(e) + \Qns(f,f)(e),
\end{equation}
beginning with the more straightforward estimate of $\Qns$. The assumptions on $f$ guarantee
\begin{equation}\label{e.uf_upper_bound}
    0 \le f(w) \le \abs{w}^{-m}, \qquad f(e) = 1, \qquad \text{and} \qquad f(w) = 0 \;\;\text{if }w\in B_\delta.
\end{equation}
Recalling the form of $\Qns$ from \eqref{e.Qns}, we immediately deduce
\begin{equation}\label{e.Qns_estimate}
    \Qns(f,f)(e) = C_b f(e) \int_{\R^d} f(w) \abs{e-w}^\gamma \dd w \lesssim \int_{\R^d\setminus B_\delta} \abs{w}^{-m}\abs{e-w}^{\gamma} \dd w
    \lesssim 1,
\end{equation}
where the constant depends only on $m$, $\delta$, and $\QB$.

We now estimate $\Qs$, which is a bit more delicate. Recalling the form of $\Qs$ from \eqref{e.Qs}, $\Qs(f,f)(e)$ takes the form
\begin{align*}
    \Qs(f,f)(e) &= \int_{\R^d} f(v_*')\left( \int_{v'_*-e \perp v'-e}  \left(f(v') - f(e)\right) B_2(|e-v_*'|, |e-v'|) \dd v' \right) \dd v_*' 
\end{align*}
The resulting integral has a potential non-integrable singularity only when $v' = e$, corresponding to $\theta = 0$ in the original formulation of the Boltzmann equation \eqref{eq:Boltzmann}. Consequently, we analyze $v'$ near $e$ and $v'$ far from $e$ separately.

We begin with the case when $v'$ is close to $e$. Using that $f(v) - \abs{v}^{-m}$ obtains its maximum of $0$ at $e$, the second order Taylor expansion of $|v|^{-m}$ gives:
\begin{equation}
     f(v') - f(e) \le |v'|^{-m} - 1 = -m e \cdot (v'-e) + O(|v'-e|^2).
\end{equation}
The first order term is odd with respect to $v'-e$ and vanishes in the integral. Using the estimate~\eqref{eq:B_2 condition} on $B_2$, we find
\begin{equation}\label{e.c1}
\begin{split}
\int_{\substack{v'-e \perp v_*'-e \\ |v'-e|\leq1/2 }}\limits  \left(f(v') - f(e) \right) B_2(|e-v_*'|, |e-v'|) \dd v'
    &\lesssim \int_{\substack{v'-e \perp v_*'-e \\ |v'-e|\leq 1/2 }}\limits  |v'-e|^2 B_2(|e-v_*'|, |e-v'|) \dd v'
    \\ &\lesssim 1 + |v'_*-e|^{\gamma+2}.
\end{split}
\end{equation}

Away from the potential singularity, for $|v'-e|>1/2$, we drop the negative term and estimate more coarsely. Recall that $f(v) = 0$ for $|v| < \delta$, so that on this domain $f(v') \lesssim \langle v'\rangle^{-m}$.  Hence,
\begin{equation}\label{e.c3}
\begin{split}
\int_{\substack{v'-e \perp v_*'-e \\ |v'-e| > 1/2 }}\limits  &\left( f(v') - f(e)\right) B_2(|e-v_*'|, |e-v'|) \dd v'
    \leq \int_{\substack{v'-e \perp v_*'-e \\ |v'-e| > 1/2 }}\limits  \langle v'\rangle^{-m} B_2(|e-v_*'|, |e-v'|) \dd v'
    \\&
    = \int_{\substack{v'-e \perp v_*'-e \\ |v'-e| > 1/2 }}\limits  \frac{r^{3+\gamma}}{|e- v_*'||v'-e|^2 \langle v'\rangle^{m}} \left(\frac{|v'-e|^2|v_*'-e|}{r^{3+\gamma}} B_2(|e-v_*'|, |e-v'|)\right) \dd v'
    \\&
    \lesssim \max_{v'}\left(\frac{r^{3+\gamma}}{|e- v_*'| \langle v'\rangle^{m+2}} \right)\int_{\substack{v'-e \perp v_*'-e \\ |v'-e| > 1/2 }}\limits   \left(\frac{|v'-e|^2|v_*'-e|}{r^{3+\gamma}} B_2(|e-v_*'|, |e-v'|)\right) \dd v'.
\end{split}
\end{equation}
Recalling that $|v'-e|^2 + |v_*'-e|^2 = r^2$ on the hyperplane $v'-e\perp v'_* - e$, we see that
\begin{equation}
    \max_{v'}\left(\frac{r^{3+\gamma}}{|e- v_*'| \langle v'\rangle^{m+2}} \right)
    \lesssim \frac{1}{|v_*'-e|} + |v_*'-e|^{2+\gamma}.
\end{equation}
Combining this with the integrability condition~\eqref{eq:B_2 condition} on $B_2$ and the inequality~\eqref{e.c3}, we find
\begin{equation}\label{e.c2}
    \int_{\substack{v'-e \perp v_*'-e \\ |v'-e| > 1/2 }}\limits  \left( f(v') - f(e)\right) B_2(|e-v_*'|, |e-v'|) \dd v'
    \lesssim \frac{1}{|v_*'-e|} + |v_*'-e|^{2+\gamma}.
\end{equation}

Combining the bounds~\eqref{e.c1} and~\eqref{e.c2}, we obtain
\[
    Q_s(f,f)(e)
    \lesssim \int_{\R^d} f(v'_*) (|v_*'-e|^{-1} + |v'_*-e|^{\gamma+2} ) \dd v'_* \lesssim 1.
\]
This final integral is bounded since
\[ f(v_*') \leq \begin{cases}
    |v_*'|^{-m} & \text{for } |v'_*| > \delta \\
    0 & \text{for } |v'_*| \leq \delta.
 \end{cases} \]
This, along with~\eqref{e.Qns_estimate}, completes the proof.
\end{proof}

\section{Deriving
Theorem \ref{thm:main_continuation} from known results and Theorem \ref{thm:main_propagation}}\label{sec:continuation}

\begin{proof}[Proof of Theorem \ref{thm:main_continuation}]
This is immediate from Theorem \ref{t:local-well-posedness-Landau-soft} and Theorem \ref{t:local-well-posedness-Boltzmann-soft}.  Indeed, by assumption, the $L^1_t L^\infty_{d+\gamma}$-norm of $f$ is controlled and Theorem \ref{thm:main_propagation} implies that the $L^\infty_t L^\infty_m$-norm of $f$ is controlled. The above-cited continuation criteria then apply.
\end{proof}

\begin{proof}[Proof of Corollary \ref{cor:lower_bound_blowup_rate}]

    Fix $m > m_0$ and $C = C(m,Q)$ the corresponding constant from Theorem \ref{thm:main_propagation}. Define $y(t) = \norm{f(t)}_{L^\infty_m}$ for some $m > m_0$. Then, using the translation invariance (in time) of \eqref{eq:kinetic}, Theorem \ref{thm:main_propagation} yields an integral inequality for $y$: \begin{equation}
        \text{for any }0<s<t<T, \qquad y(t) \le y(s) \exp\left(C\int_s^t \norm{f(\tau)}_{L^\infty_{d+\gamma}} \dd \tau\right) \le y(s) \exp\left(C\int_s^t y(\tau) \dd \tau\right).
    \end{equation}
    Taking a difference quotient and passing to the limit (using the qualitative regularity of $f$), yields the following Riccati-type differential inequality for $y$:
    \begin{equation}
        \text{for any }0 < \tau < T, \qquad y'(\tau) \le Cy(\tau)^2.
    \end{equation}
    Solving the associated differential equation, we obtain
    \begin{equation}
     \text{for any }0 < s < t < T, \qquad -\frac{1}{y(t)} + \frac{1}{y(s)} \le C(t-s).
    \end{equation}
    Using the assumption that $\displaystyle \limsup_{t\to T^-} y(t) = +\infty$, we conclude
    \begin{equation}
        y(s) \ge \frac{1}{C(T-s)}.
    \end{equation}
    The proof is complete.

\end{proof}

\bibliographystyle{abbrv}
\bibliography{LandauRefs}

@article {AMUXY4,
    AUTHOR = {Alexandre, Radjesvarane and Morimoto, Yoshinori and Ukai,
              Seiji and Xu, Chao-Jiang and Yang, Tong},
     TITLE = {Local existence with mild regularity for the {B}oltzmann
              equation},
   JOURNAL = {Kinet. Relat. Models},
  FJOURNAL = {Kinetic and Related Models},
    VOLUME = {6},
      YEAR = {2013},
    NUMBER = {4},
     PAGES = {1011--1041},
      ISSN = {1937-5093,1937-5077},
   MRCLASS = {35Q20 (35A01 35B65 76P05 82C40)},
  MRNUMBER = {3177640},
MRREVIEWER = {Linjie\ Xiong},
       DOI = {10.3934/krm.2013.6.1011},
       URL = {https://doi.org/10.3934/krm.2013.6.1011},
}

@article {AMUXY3,
    AUTHOR = {Alexandre, Radjesvarane and Morimoto, Yoshinori and Ukai,
              Seiji and Xu, Chao-Jiang and Yang, Tong},
     TITLE = {Uniqueness of solutions for the non-cutoff {B}oltzmann
              equation with soft potential},
   JOURNAL = {Kinet. Relat. Models},
  FJOURNAL = {Kinetic and Related Models},
    VOLUME = {4},
      YEAR = {2011},
    NUMBER = {4},
     PAGES = {919--934},
      ISSN = {1937-5093,1937-5077},
   MRCLASS = {35Q20 (35A02 35D30)},
  MRNUMBER = {2861580},
MRREVIEWER = {Xianpeng\ Hu},
       DOI = {10.3934/krm.2011.4.919},
       URL = {https://doi.org/10.3934/krm.2011.4.919},
}

@article {AMUXY1,
    AUTHOR = {Alexandre, Radjesvarane and Morimoto, Yoshinori and Ukai,
              Seiji and Xu, Chao-Jiang and Yang, Tong},
     TITLE = {Regularizing effect and local existence for the non-cutoff
              {B}oltzmann equation},
   JOURNAL = {Arch. Ration. Mech. Anal.},
  FJOURNAL = {Archive for Rational Mechanics and Analysis},
    VOLUME = {198},
      YEAR = {2010},
    NUMBER = {1},
     PAGES = {39--123},
      ISSN = {0003-9527,1432-0673},
   MRCLASS = {82C40 (35Q20 76P05)},
  MRNUMBER = {2679369},
MRREVIEWER = {Silvia\ Lorenzani},
       DOI = {10.1007/s00205-010-0290-1},
       URL = {https://doi.org/10.1007/s00205-010-0290-1},
}

@article {ImbertMouhotSilvestre_decay,
    AUTHOR = {Imbert, Cyril and Mouhot, Cl\'ement and Silvestre, Luis},
     TITLE = {Decay estimates for large velocities in the {B}oltzmann
              equation without cutoff},
   JOURNAL = {J. \'Ec. polytech. Math.},
  FJOURNAL = {Journal de l'\'Ecole polytechnique. Math\'ematiques},
    VOLUME = {7},
      YEAR = {2020},
     PAGES = {143--184},
      ISSN = {2429-7100,2270-518X},
   MRCLASS = {35Q20 (35B40 76P05)},
  MRNUMBER = {4033752},
MRREVIEWER = {Bertrand\ Lods},
       DOI = {10.5802/jep.113},
       URL = {https://doi.org/10.5802/jep.113},
}

@incollection {Golse_hydrodynamic,
    AUTHOR = {Golse, Fran\c{c}ois},
     TITLE = {Hydrodynamic limits},
 BOOKTITLE = {European {C}ongress of {M}athematics},
     PAGES = {699--717},
 PUBLISHER = {Eur. Math. Soc., Z\"urich},
      YEAR = {2005},
      ISBN = {3-03719-009-4},
   MRCLASS = {82C40 (35F20 35Q35 76-02 76A02 76P05 82-02)},
  MRNUMBER = {2185776},
}

@article{li2026local_hardpotentials,
  title={Local well-posedness for the {B}oltzmann equation with hard potentials},
  author={Li, Hao-Guang and Li, Wei-Xi and Xu, Chao-Jiang},
  journal={arXiv preprint arXiv:2602.19148},
  year={2026}
}

@article {HST_Boltzmann_existence,
    AUTHOR = {Henderson, Christopher and Snelson, Stanley and Tarfulea,
              Andrei},
     TITLE = {Classical solutions of the {B}oltzmann equation with irregular
              initial data},
   JOURNAL = {Ann. Sci. \'Ec. Norm. Sup\'er. (4)},
  FJOURNAL = {Annales Scientifiques de l'\'Ecole Normale Sup\'erieure.
              Quatri\`eme S\'erie},
    VOLUME = {58},
      YEAR = {2025},
    NUMBER = {1},
     PAGES = {107--201},
      ISSN = {0012-9593,1873-2151},
   MRCLASS = {35Q20 (35A09 82C40)},
  MRNUMBER = {4895262},
}

@article {SnelsonTaylor,
    AUTHOR = {Snelson, Stanley and Taylor, Shelly Ann},
     TITLE = {Existence of smooth solutions to the {L}andau equation with
              hard potentials and irregular initial data},
   JOURNAL = {J. Stat. Phys.},
  FJOURNAL = {Journal of Statistical Physics},
    VOLUME = {192},
      YEAR = {2025},
    NUMBER = {10},
     PAGES = {Paper No. 142},
      ISSN = {0022-4715,1572-9613},
   MRCLASS = {35 (82)},
  MRNUMBER = {4970362},
       DOI = {10.1007/s10955-025-03525-7},
       URL = {https://doi.org/10.1007/s10955-025-03525-7},
}

@article{BardosGolseLevermore1,
    author = "Bardos, C. and Golse, F. and Levermore, D.",
    title = "Fluid dynamic limits of kinetic equations. {I}. {F}ormal derivations",
    journal = "J. Statist. Phys.",
    year = "1991",
    volume = "63",
    number = "1-2",
    pages = "323-344"
}

@article{BardosGolseLevermore2,
    author = "Bardos, C. and Golse, F. and Levermore, D.",
    title = "Fluid dynamic limits of kinetic equations. {II}. {C}onvergence proofs for the {B}oltzmann equation",
    journal = "Comm. Pure and Appl. Math.",
    year = "1993",
    volume = "46",
    number = "5",
    pages = "667-753"
}

@article{Caflisch,
    author = "Caflisch, R.",
    title = "The fluid dynamic limit of the nonlinear {B}oltzmann equation",
    journal = "Comm. Pure Appl. Math.",
    year = "1980",
    volume = "33",
    number = "5",
    pages = "651-666"
}

@article {Silvestre_new_regularization,
    AUTHOR = {Silvestre, Luis},
     TITLE = {A new regularization mechanism for the {B}oltzmann equation
              without cut-off},
   JOURNAL = {Comm. Math. Phys.},
  FJOURNAL = {Communications in Mathematical Physics},
    VOLUME = {348},
      YEAR = {2016},
    NUMBER = {1},
     PAGES = {69--100},
      ISSN = {0010-3616,1432-0916},
   MRCLASS = {45K05 (35B65 35Q20)},
  MRNUMBER = {3551261},
MRREVIEWER = {Cecilia\ Cavaterra},
       DOI = {10.1007/s00220-016-2757-x},
       URL = {https://doi.org/10.1007/s00220-016-2757-x},
}

@misc{Chen,
    year = "2023",
    author = "Chen, J.",
    title = "Nearly self-similar blowup of the slightly perturbed homogeneous {L}andau equation with very soft potentials",
    howpublished = "ar{X}iv:2311.11511"
}

@article{HendersonSnelson,
  title = "${C}^\infty$ smoothing for weak solutions of the inhomogeneous {L}andau equation",
  author = "Christopher Henderson and Stanley Snelson",
  journal = "Arch. Ration. Mech. Anal.",
  year= "2020",
  volume = "236",
  number = "1",
  pages = "113-143"
}

@article{HendersonSnelsonTarfulea,
    title = "Local solutions of the {L}andau equation with rough, slowly decaying initial datum",
    author = "Henderson, C. and Snelson, S. and Tarfulea, A.",
    journal = "Ann. Inst. H. Poincar{\'e} C Anal. Non Lin{\'e}aire",
    year = "2020",
    pages = "1345-1377",
    number = "6",
    volume = "37"
}

@article{HendersonSnelsonTarfulea1,
    AUTHOR = {Henderson, Christopher and Snelson, Stanley and Tarfulea,
              Andrei},
     TITLE = {Local existence, lower mass bounds, and a new continuation
              criterion for the {L}andau equation},
   JOURNAL = {J. Differential Equations},
  FJOURNAL = {Journal of Differential Equations},
    VOLUME = {266},
      YEAR = {2019},
    NUMBER = {2-3},
     PAGES = {1536--1577},
      ISSN = {0022-0396,1090-2732},
   MRCLASS = {35Q20 (35A01 35B44)},
  MRNUMBER = {3906224},
       DOI = {10.1016/j.jde.2018.08.005},
       URL = {https://doi.org/10.1016/j.jde.2018.08.005},
}

@article{SnelsonSolomon,
    author = "Snelson, S. and Solomon, C.",
    title = "A continuation criterion for the {L}andau equation with very soft and {C}oulomb potentials",
    journal = "Arch. Ration. Mech. Anal.",
    volume = "250",
    year = "2026" 
}

@article{JangLiuSchrecker,
    author = "Jang, J. and Liu, J. and Schrecker, M.",
    title = "Converging/diverging self-similar shock waves: from collapse to reflection",
    journal = "SIAM J. Math. Anal.",
    year = "2025",
    volume = "57",
    number = "1",
    pages = "190-232"
}

@article{JangLiuSchrecker2,
    author = "Jang, J. and Liu, J. and Schrecker, M.",
    title = "On self-similar converging shock waves",
    journal = "Arch. Ration. Mech. Anal.",
    year = "2025",
    volume = "249",
    number = "3",
    pages = "83 pp."
}

@article{Guderley,
    author = "Guderley, G.",
    title = "Starke kugelige und zylindrische Verdichtungsst{{\"o}}sse in der N{{\"a}}he des Kugelmittelpunktes bzw. der Zylinderachse",
    journal = "Luftfahrtforschung",
    year = "1942",
    volume = "19",
    pages = "302 - 311"
}

@article{MerleRaphaelRodnianskiSzeftel,
    author = "Merle, F. and Rapha{\"{e}}l, P. and Rodnianski, I. and Szeftel, J.",
    title = "On the implosion of a compressible fluid {I}: {S}mooth self-similar inviscid profiles",
    journal = "Ann. of Math. (2)",
    year = "2022",
    number = "2",
    volume = "196",
    pages = "567-778"
}

@article{MerleRaphaelRodnianskiSzeftel2,
    author = "Merle, F. and Rapha{\"{e}}l, P. and Rodnianski, I. and Szeftel, J.",
    title = "On the implosion of a compressible fluid {II}: {S}ingularity formation",
    journal = "Ann. of Math. (2)",
    year = "2022",
    number = "2",
    volume = "196",
    pages = "779-889"
}

@article{CialdeaShkollerVicol,
  title={Classical {E}uler flows generate the strong {G}uderley imploding shock wave},
  author={Cialdea, G. and Shkoller, S. and Vicol, V.},
  journal={arXiv:2510.19688},
  year={2025}
}

@article {Lazarus,
    AUTHOR = {Lazarus, Roger B.},
     TITLE = {Self-similar solutions for converging shocks and collapsing
              cavities},
   JOURNAL = {SIAM J. Numer. Anal.},
  FJOURNAL = {SIAM Journal on Numerical Analysis},
    VOLUME = {18},
      YEAR = {1981},
    NUMBER = {2},
     PAGES = {316--371},
      ISSN = {0036-1429},
   MRCLASS = {76L05},
  MRNUMBER = {612145},
MRREVIEWER = {R.\ M.\ Gundersen},
       DOI = {10.1137/0718022},
       URL = {https://doi.org/10.1137/0718022},
}

@article{BuckmasterCaoGomez,
    AUTHOR = {Buckmaster, Tristan and Cao-Labora, Gonzalo and
              G\'omez-Serrano, Javier},
     TITLE = {Smooth imploding solutions for 3{D} compressible fluids},
   JOURNAL = {Forum Math. Pi},
  FJOURNAL = {Forum of Mathematics. Pi},
    VOLUME = {13},
      YEAR = {2025},
     PAGES = {Paper No. e6, 139},
      ISSN = {2050-5086},
   MRCLASS = {35Q30 (35B35 35Q35 65G30 76N10)},
  MRNUMBER = {4862915},
MRREVIEWER = {Gudrun\ Th\"ater},
       DOI = {10.1017/fmp.2024.12},
       URL = {https://doi.org/10.1017/fmp.2024.12},
}

@article{Chaturvedi,
    author = "Chaturvedi, S.",
    title = "Local existence for the {L}andau equation with hard potentials",
    journal = "SIAM J. Math. Anal.",
    year = "2023",
    volume = "55",
    number = "5",
    pages = "5345-5385"
}

@incollection {Silvestre_Review,
    AUTHOR = {Silvestre, Luis},
     TITLE = {Regularity estimates and open problems in kinetic equations},
 BOOKTITLE = {{{$\rm A^3N^2M$}}: approximation, applications, and analysis
              of nonlocal, nonlinear models},
    SERIES = {IMA Vol. Math. Appl.},
    VOLUME = {165},
     PAGES = {101--148},
 PUBLISHER = {Springer, Cham},
      YEAR = {2023},
      ISBN = {978-3-031-34088-8; 978-3-031-34089-5},
   MRCLASS = {35Q20 (35B65 76N10 76P05)},
  MRNUMBER = {4807233},
       DOI = {10.1007/978-3-031-34089-5\_3},
       URL = {https://doi.org/10.1007/978-3-031-34089-5_3},
}

@article {nash1958continuity,
    AUTHOR = {Nash, J.},
     TITLE = {Continuity of solutions of parabolic and elliptic equations},
   JOURNAL = {Amer. J. Math.},
  FJOURNAL = {American Journal of Mathematics},
    VOLUME = {80},
      YEAR = {1958},
     PAGES = {931--954},
      ISSN = {0002-9327,1080-6377},
   MRCLASS = {35.00},
  MRNUMBER = {100158},
MRREVIEWER = {C.\ B.\ Morrey, Jr.},
       DOI = {10.2307/2372841},
       URL = {https://doi.org/10.2307/2372841},
}

@article {imbert2022global,
    AUTHOR = {Imbert, Cyril and Silvestre, Luis},
     TITLE = {Global regularity estimates for the {B}oltzmann equation
              without cut-off},
   JOURNAL = {J. Amer. Math. Soc.},
  FJOURNAL = {Journal of the American Mathematical Society},
    VOLUME = {35},
      YEAR = {2022},
    NUMBER = {3},
     PAGES = {625--703},
      ISSN = {0894-0347,1088-6834},
   MRCLASS = {35R09 (35Q20)},
  MRNUMBER = {4433077},
       DOI = {10.1090/jams/986},
       URL = {https://doi.org/10.1090/jams/986},
}

@article {imbert2020regularity,
    AUTHOR = {Imbert, Cyril and Silvestre, Luis},
     TITLE = {Regularity for the {B}oltzmann equation conditional to
              macroscopic bounds},
   JOURNAL = {EMS Surv. Math. Sci.},
  FJOURNAL = {EMS Surveys in Mathematical Sciences},
    VOLUME = {7},
      YEAR = {2020},
    NUMBER = {1},
     PAGES = {117--172},
      ISSN = {2308-2151,2308-216X},
   MRCLASS = {35Q20 (35B65 35H10 35R09)},
  MRNUMBER = {4195746},
       DOI = {10.4171/emss/37},
       URL = {https://doi.org/10.4171/emss/37},
}

@article {henderson2020self,
    AUTHOR = {Henderson, Christopher and Snelson, Stanley and Tarfulea,
              Andrei},
     TITLE = {Self-generating lower bounds and continuation for the
              {B}oltzmann equation},
   JOURNAL = {Calc. Var. Partial Differential Equations},
  FJOURNAL = {Calculus of Variations and Partial Differential Equations},
    VOLUME = {59},
      YEAR = {2020},
    NUMBER = {6},
     PAGES = {Paper No. 191, 13},
      ISSN = {0944-2669,1432-0835},
   MRCLASS = {35Q20 (35B60 35B65)},
  MRNUMBER = {4163318},
       DOI = {10.1007/s00526-020-01856-9},
       URL = {https://doi.org/10.1007/s00526-020-01856-9},
}

@article {cao2025non,
    AUTHOR = {Cao-Labora, G. and G\'omez-Serrano, J. and Shi, J. and
              Staffilani, G.},
     TITLE = {Non-radial implosion for compressible {E}uler and
              {N}avier-{S}tokes in {$\mathbb T^3$} and {$\mathbb R^3$}},
   JOURNAL = {Camb. J. Math.},
  FJOURNAL = {Cambridge Journal of Mathematics},
    VOLUME = {13},
      YEAR = {2025},
    NUMBER = {4},
     PAGES = {753--885},
      ISSN = {2168-0930,2168-0949},
   MRCLASS = {35Q30 (35Q31 35R01)},
  MRNUMBER = {5012344},
       DOI = {10.4310/cjm.260107044458},
       URL = {https://doi.org/10.4310/cjm.260107044458},
}

@article {buckmaster2023smooth,
    AUTHOR = {Buckmaster, Tristan and Cao-Labora, Gonzalo and
              G\'omez-Serrano, Javier},
     TITLE = {Smooth self-similar imploding profiles to 3{D} compressible
              {E}uler},
   JOURNAL = {Quart. Appl. Math.},
  FJOURNAL = {Quarterly of Applied Mathematics},
    VOLUME = {81},
      YEAR = {2023},
    NUMBER = {3},
     PAGES = {517--532},
      ISSN = {0033-569X,1552-4485},
   MRCLASS = {76N10 (35Q31)},
  MRNUMBER = {4623212},
       DOI = {10.1090/qam/1661},
       URL = {https://doi.org/10.1090/qam/1661},
}

@article{bedrossian2026finite,
  title={Finite time singularities in the {L}andau equation with very hard potentials},
  author={Bedrossian, Jacob and Chen, Jiajie and Gualdani, Maria Pia and Ji, Sehyun and Vicol, Vlad and Yang, Jincheng},
  journal={arXiv:2602.05981},
  year={2026}
}

@article{shao2025blow,
  title={Blow-up of the 3-D compressible {N}avier-{S}tokes equations for monatomic gases},
  author={Shao, Feng and Wei, Dongyi and Wang, Shumao and Zhang, Zhifei},
  journal={arXiv:2501.15701},
  year={2025}
}

@article {henderson2025decay,
    AUTHOR = {Henderson, Christopher and Snelson, Stanley and Tarfulea,
              Andrei},
     TITLE = {Decay estimates and continuation for the non-cutoff
              {B}oltzmann equation},
   JOURNAL = {Math. Ann.},
  FJOURNAL = {Mathematische Annalen},
    VOLUME = {392},
      YEAR = {2025},
    NUMBER = {4},
     PAGES = {4739--4771},
      ISSN = {0025-5831,1432-1807},
   MRCLASS = {35Q20 (35B40 82C40)},
  MRNUMBER = {4958489},
       DOI = {10.1007/s00208-025-03207-5},
       URL = {https://doi.org/10.1007/s00208-025-03207-5},
}

@article {JenssenTsikkou1,
    AUTHOR = {Jenssen, Helge Kristian and Tsikkou, Charis},
     TITLE = {Amplitude blowup in radial isentropic {E}uler flow},
   JOURNAL = {SIAM J. Appl. Math.},
  FJOURNAL = {SIAM Journal on Applied Mathematics},
    VOLUME = {80},
      YEAR = {2020},
    NUMBER = {6},
     PAGES = {2472--2495},
      ISSN = {0036-1399,1095-712X},
   MRCLASS = {35L45 (35L67 35Q31 76N10)},
  MRNUMBER = {4181105},
       DOI = {10.1137/20M1340241},
       URL = {https://doi.org/10.1137/20M1340241},
}

@article{Jenssen,
      title={Amplitude blowup in compressible {E}uler flows without shock formation}, 
      author={Helge Kristian Jenssen},
      year={2025},
      journal={arXiv:2501.09037}, 
}

@article {serre,
    AUTHOR = {Serre, Denis},
     TITLE = {Divergence-free positive symmetric tensors and fluid dynamics},
   JOURNAL = {Ann. Inst. H. Poincar\'e{} C Anal. Non Lin\'eaire},
  FJOURNAL = {Annales de l'Institut Henri Poincar\'e{} C. Analyse Non
              Lin\'eaire},
    VOLUME = {35},
      YEAR = {2018},
    NUMBER = {5},
     PAGES = {1209--1234},
      ISSN = {0294-1449,1873-1430},
   MRCLASS = {76N10 (35B45 35Q20 35Q31 35Q35 53A45)},
  MRNUMBER = {3813963},
MRREVIEWER = {Giovanni\ Franco\ Crosta},
       DOI = {10.1016/j.anihpc.2017.11.002},
       URL = {https://doi.org/10.1016/j.anihpc.2017.11.002},
}

@article{Hunter,
    title={On the collapse of an empty cavity in water},
    volume={8},
    DOI={10.1017/S0022112060000578},
    number={2},
    journal={Journal of Fluid Mechanics},
    author={Hunter, C.},
    year={1960},
    pages={241–263}
}

@article{chen2026smoothstableeulerimplosions,
      title={Smooth and stable {E}uler implosions}, 
      author={Jiajie Chen and Steve Shkoller and Vlad Vicol},
      year={2026},
      journal={arXiv:2605.00808}
}

\end{document}